\documentclass[compress,final,3p,times,12.8pt]{elsarticle}
\usepackage{mathrsfs}
\usepackage{amsfonts}

\usepackage{graphics}
 \usepackage{color}
 \usepackage{array}
\usepackage{amssymb}
\usepackage{amsmath}
\usepackage{diagbox}
\usepackage{subfigure}
\usepackage{multirow}

\newtheorem{thm}{Theorem}
\newtheorem{Assu}{Assumption}

\newtheorem{lem}{Lemma}
\newtheorem{rem}[thm]{Remark}

\numberwithin{equation}{section}
 \numberwithin{lem}{section}
 \numberwithin{Defi}{section}
 \numberwithin{thm}{section}
 \numberwithin{Rem}{section}
  \numberwithin{Coro}{section}
  \numberwithin{Fig}{section}

\begin{document}

\begin{frontmatter}

\title{Data-driven method to learn the most probable transition pathway and stochastic differential equation} \tnotetext[label1]{This work is
 supported by NSFC (Grant Nos. 11771162 and 12141107)}

\author{Xiaoli Chen\fnref{addr4}}
\ead{xlchen@nus.edu.sg}
\author{Jinqiao Duan\fnref{addr3}}
\ead{duan@iit.edu}
\author{Jianyu Hu\corref{cor1}\fnref{addr1}}
\ead{jianyuhu@hust.edu.cn}
\author{Dongfang Li\fnref{addr1,addr2}}
\ead{hustldf@gmail.com}\cortext[cor1]{Corresponding author}


\address[addr1]{School of Mathematics and Statistics, Huazhong University of Science and Technology, Wuhan 430074,  China}

\address[addr2]{Department of Mathematics, National University of Singapore, 119077, Singapore}

\address[addr3]{Hubei Key Laboratory of Engineering Modeling and Scientific Computing, Huazhong University of Science and Technology, Wuhan 430074, China}
\address[addr4]{Department of Applied Mathematics, College of Computing, Illinois Institute of Technology, Chicago, IL 60616, USA}

\date{\today}

\begin{abstract}
{Transition phenomena between metastable states play an important role in complex systems due to noisy fluctuations. 
We introduce the Onsager-Machlup theory and the Freidlin-Wentzell theory to quantify rare events in stochastic differential equations. By the variational principle, the most probable transition pathway is the minimizer of the action functional, which is governed by the Euler-Lagrange equation. 
In this paper, the physics-informed neural networks (PINNs) are presented to compute the most probable transition pathway through computing the Euler-Lagrange equation. 
The convergence result for the empirical loss is obtained for the forward problem. 
Then we investigate the inverse problem to extract the stochastic differential equation from the most probable transition pathway data. 
Finally, several numerical experiments are presented to verify the effectiveness of our methods.
 }
\end{abstract}
 
\begin{keyword}
{physics-informed neural networks; most probable transition pathway; Euler-Lagrange equation; Markovian bridge process; inverse problem}
\end{keyword}

\end{frontmatter}

\section{Introduction}
\textcolor{red}{Stochastic dynamical systems are used to model complicated phenomena} in the physical, chemical, and biological sciences\cite{arnold2013random,duan2015introduction,duan2014effective,imkeller2012stochastic}. Noisy fluctuations result in transitions between metastable states \cite{budhiraja2019analysis,ditlevsen1999observation}, which are impossible in deterministic systems. 
The Onsager-Machlup theory and the Freidlin-Wentzell theory are two essential tools to study the rare events of stochastic dynamical systems, which have been studied with a rich history; see for examples \cite{capitaine2000onsager,freidlin2012random,ikeda2014stochastic,varadhan1984large}.
They induce an action functional to quantify the probability of solution paths on a small tube and provide information about system transitions.
The minimum value of the action functional means the largest probability of the path tube, and the minimizer is the most probable transition pathway that is governed by the Euler-Lagrange equation.
Comparing to samples of stochastic differential equations, the transition paths contain more important information, which makes them effective for extracting the stochastic differential equation. In this way, the Onsager-Machlup theory and the Freidlin-Wentzell theory provide a new perspective on inferring the stochastic differential equation.

There are some numerical methods for calculating the most probable transition pathway. In the Onsager-Machlup framework, D\"{u}rr and Bach \cite{durr1978onsager} regarded the action functional as an integral of a Lagrangian for determining the most probable transition pathway of a diffusion process by the variational principle. Chao and Duan \cite{chao2019onsager} investigated the most probable transition pathway for scalar stochastic dynamical systems with L\'{e}vy noise by the Euler-Lagrange equation and shooting method. Hu and Chen \cite{hu2021transition} used a neural shooting method to compute the most probable transition pathway. 
For the small noise case, the Freidlin-Wentzell theory of large deviations provides a framework to understand the dynamics of transition phenomena and the description of the effect of small random perturbations \cite{freidlin2012random,varadhan1984large}. 
Numerical methods have been developed to compute the most probable transition pathway, 
including the string method \cite{weinan2002string,weinan2003energy}, the minimum action method \cite{weinan2004minimum,wan2018convergence,wan2021minimum}, the adaptive minimum action method \cite{wan2018hp,zhou2008adaptive} and the geometric minimum action method \cite{Heymann2008TheGM}. 
In high dimensions, computing will become very complex and challenging.

It is remarkable that our viewpoint sheds some new light on recovering stochastic differential equations from transition paths. In the previous work, the researchers used the solution samples.
Opper et.al. \cite{batz2016variational} recovered the drift function through the variational estimation method.
In \cite{boninsegna2018sparse,ren2020identifying}, the authors used sparse learning to extract the stochastic differential equation.
Chen et.al. \cite{chen2021solving} learned the stochastic differential equation from discrete particle samples at different time snapshots using the Fokker-Planck equation and neural network. In \cite{yang2020generative}, Yang et.al. used samples from a few snapshots of unpaired data to infer stochastic differential equations. 
Dietrich et.al. \cite{dietrich2021learning} 
proposed a maximal likelihood estimation to learn the stochastic differential equation with trajectories observation data. 
These methods require a lot of data about the solution samples.

Benefiting from the recent development of machine learning, the physics-informed neural networks (PINNs) method \cite{raissi2019physics} is successfully applied to solve differential equations and inverse problems \cite{chen2021learning,Lou2021PhysicsinformedNN,lu2021deepxde,raissi2020hidden}. Shin et.al. \cite{shin2020convergence} proved the convergence of the PINNs method to solve the linear second-order elliptic and parabolic type partial differential equations. Combining with the PINNs method, we use the Onsager-Machlup theory and the Freidlin-Wentzell theory to compute the most probable transition pathway and infer the stochastic differential equation. 
In contrast to the PINNs method for solving partial differential equations \cite{shin2020convergence}, we prove a convergence result of the PINNs method to solve the nonlinear Euler-Lagrange equation with initial and final states. 
Under the probabilistic space filling arguments, we bound the expected PINNs loss in terms of the empirical PINNs loss. We also show the convergence result for the empirical loss, which provides a theoretical support for computing the most probable transition pathway. 

For the inverse problem, we use the neural network to extract the drift function of the stochastic differential equation in both the Onsager-Machlup framework and the Freidlin-Wentzell framework from the observation data. The observation data are transition paths of the stochastic differential equation, sampled by the corresponding Markovian bridge process \cite{delarue2017ab,orland2011generating} in one dimension. And in high dimensions, the transition paths can not be easily simulated, as we know. Instead, we use the most transition pathways computed by the PINNs method with Gaussian noise perturbations as the observation data. As a comparison, we also extract the drift function from the observation data computed by the PINNs method.

In this paper, we present a new viewpoint to compute the most transition pathways and extract stochastic differential equations. It is convenient to calculate the most transition pathways using the PINNs method. 
\textcolor{red}{We demonstrate the convergence of the PINNs method for solving the nonlinear Euler-Lagrange equation.}
It is shown that the expected loss is bound by the empirical loss. To extract the stochastic differential equation, both parametric and non-parametric methods are presented from the transition paths data rather than solution samples.
Our method needs less data to infer the stochastic differential equation, comparing with existing methods \cite{batz2016variational,boninsegna2018sparse,chen2021solving}.

The remainder of this paper is structured as follows. In Section 2, we introduce the Onsager-Machlup theory and the Freidlin-Wentzell theory to quantify the rare events of stochastic dynamical systems. Then, by the variational principle, we apply the corresponding Euler-Lagrange equation to compute the most probable transition pathway and introduce the Markovian bridge process to sample such rare events. In Section 3, we prove a convergence result of the PINNs method to solve the Euler-Lagrange equation. Moreover, we present both parametric and non-parametric methods to extract the drift function through neural networks. Some numerical experiments are presented in Section 4, followed by Section 5 with a conclusion.

\section{Mathematical setup}
Consider the following stochastic differential equation
\begin{equation}
\begin{split}\label{SDE_model}
dX(t)=f(X(t))dt+\varepsilon\sigma dW(t),
\end{split}
\end{equation}
for $t\in[0,T]$, with initial data $X(0)=x_0 \in \mathbb{R}^d$, where $f:\mathbb{R}^d\rightarrow\mathbb{R}^d$ is a drift function, $\sigma$ is a $d\times k$ matrix, $W$ is a Brownian motion in $\mathbb{R}^k$, and $\varepsilon$ is a positive parameter. Here, $\varepsilon$ is considered since we will apply the Freidlin-Wentzell  large deviation theory. The well-posedness of the stochastic differential equation \eqref{SDE_model} has been widely investigated \cite{karatzas2012brownian}. If the drift and diffusion coefficients are locally Lipschitz and satisfy the ``one-sided linear growth condition," the global solution of \eqref{SDE_model} is unique. 

We are interested in the transition phenomena between two metastable states. The metastable states are the stable equilibrium states of the corresponding deterministic system. Noisy fluctuations result in transitions between metastable states, which are impossible in the corresponding deterministic system. The key is to estimate the probability of solution pathways from equation \eqref{SDE_model} in a small tube in path space. The Onsager-Machlup action functional theory and large deviation theory quantify this probability as a functional in the path space. Thus, we could minimize the functional to find the most probable transition pathway. We compute the most probable transition pathway and learn the drift coefficient from the observation data for stochastic dynamical systems using machine learning methods.

\subsection{Onsager-Machlup action functional and Freidlin-Wentzell action functional} 
To investigate transition phenomena, one should estimate the probability of the solution paths in a small tube. 
When the parameter $\varepsilon$ is a positive constant, we could use the Onsager-Machlup action functional to quantify these transition paths on a small tube of a given function.
\textcolor{red}{While for the asymptotically small parameter $\varepsilon$, the large deviation theory works well, yielding the Freidlin-Wentzell action functional.}

The Onsager-Machlup theory \cite{durr1978onsager} of stochastic dynamical system \eqref{SDE_model} gives an approximation of the probability 
\begin{equation}
\begin{split}\label{propom}
\mathbb{P}(\{\|X-z\|_T \leqslant \delta\}) \propto C(\delta,T) \exp \left\{-S^{OM}(z,\dot{z})\right\},
\end{split}
\end{equation}
where $\delta$ is positive and sufficiently small, $z$ is in the space of continuous functions in the time interval $[0,T]$ with supremum norm $\|\cdot\|_T$, and the Onsager-Machlup action functional $S^{OM}$ is defined as
\begin{equation}
\begin{split}\label{om}
S^{OM}(z,\dot{z})=\frac{1}{2}\int_0^T\left[\dot{z}-f(z)\right][\varepsilon^2\sigma\sigma^T]^{-1}[\dot{z}-f(z)]+\nabla \cdot f(z)dt.
\end{split}
\end{equation}

\textcolor{red}{For positive and sufficiently small $\delta$ and $\varepsilon$, the Freidlin-Wentzell theory \cite{freidlin2012random} of large deviations asserts}
\begin{equation}
\begin{split}\label{propfw}
\mathbb{P}(\{\|X-z\|_T \leqslant \delta\}) \propto C(\delta,T) \exp \left\{-\varepsilon^{-2}S^{FW}(z,\dot{z})\right\},
\end{split}
\end{equation}
where the Freidlin-Wentzell action functional $S^{FW}$ is defined as
\begin{equation}
\begin{split}\label{fw}
S^{FW}(z,\dot{z})=\frac{1}{2}\int_0^T\left[\dot{z}-f(z)\right][\sigma\sigma^T]^{-1}[\dot{z}-f(z)]dt.
\end{split}
\end{equation}

The most probable transition pathway connecting two metastable states is the minimizer of the Onsager-Machlup action functional or the Freidlin-Wentzell action functional. Both action functionals are successfully applied to find the most probable transition pathway for stochastic dynamical systems. The Onsager-Machlup action functional seems like a more precise approximation, since it has an additional divergence term $\nabla\cdot f$. \textcolor{red}{For sufficiently small $\varepsilon$}, the Onsager-Machlup action functional $S^{OM}$ roughly approaches to the Freidlin-Wentzell action functional $\varepsilon^{-2}S^{FW}$. But rigorously speaking, they are different, even when $\varepsilon$ tends to zero. Under a renormalization, they coincide \cite{li2021gamma}.

\subsection{Euler-Lagrange equation and the most probable transition pathway}
For simplicity, we write the action functional as the integral of a Lagrangian function 
\begin{equation}
\begin{split}\label{Lagrangian}
S(z,\dot{z})=\frac{1}{2}\int_0^TL(z,\dot{z})dt,
\end{split}
\end{equation}
where the Lagrangian $L(\cdot,\cdot)$ could be the integrand of the Onsager-Machlup action functional $S^{OM}$ or the Freidlin-Wentzell  action functional $S^{FW}$. The minimum value of the action functional \eqref{Lagrangian} means the largest probability of path tube in \eqref{propom} or \eqref{propfw}. Therefore, by minimizing the action functional \eqref{Lagrangian}, we could find the most probable transition pathway of the stochastic differential equation \eqref{SDE_model}. We emphasize that the most probable transition pathway is not required as the sample path for the stochastic system. However, it captures sample paths with the largest probability around its neighborhood.

\textcolor{red}{To this end, we limit the transition paths to functions that are twice differentiable and have fixed initial and terminal states.}
Then, by the variational principle, the most probable transition pathway of the stochastic differential equation \eqref{SDE_model} satisfies the Euler-Lagrange equation
\begin{equation}
\begin{split}\label{EL_eqn}
\frac{d}{dt}\frac{\partial}{\partial \dot{z}}L(z,\dot{z})=\frac{\partial}{\partial z}L(z,\dot{z}),
\end{split}
\end{equation}
with initial and final conditions $z(0)=x_0,~z(T)=x_T$. Here, $x_0$ and $x_T$ are often chosen as two metastable states. The existence and regularity of the minimum for the action functional \eqref{Lagrangian} are studied in \cite{chao2019onsager,freidlin2012random,hu2021transition}.

If the diffusion is a diagonal constant matrix, i.e., $\sigma=\operatorname{diag}\{\sigma_{11},...,\sigma_{dd}\}$, the Euler-Lagrange equation corresponding to the Onsager-Machlup action functional \eqref{om} reduces to
\begin{equation}
\begin{split}\label{el_om}
\ddot{z}_k=\dot{z}^j[\partial_jf^k(z)-\frac{\sigma_{kk}^2}{\sigma_{jj}^2}\partial_kf^j(z)]+\frac{\sigma_{kk}^2}{\sigma_{jj}^2}f^j(z)\partial_kf^j(z)+\frac{\varepsilon^2\sigma_{kk}^2}{2}\partial_k\partial_jf^j(z),
\end{split}
\end{equation}
for $k=1,...,d$. In addition, if $\sigma$ is an identity matrix, the Euler-Lagrange equation corresponding to the Freidlin-Wentzell action functional \eqref{fw} reduces to
\begin{equation}
\begin{split}\label{el_fw}
\ddot{z}_k=\dot{z}^j[\partial_jf^k(z)-\partial_kf^j(z)]+f^j(z)\partial_kf^j(z),
\end{split}
\end{equation}
for $k=1,...,d$. Here, we use the Einstein sum for $j$. The detailed deviation of equations \eqref{el_om} and \eqref{el_fw} is seen in Appendix B.

\subsection{The sampling of transition paths}
A transition path is a sample path of the stochastic dynamical system \eqref{SDE_model} \textcolor{red}{starting at an initial point $x_0$ and ending at the final point $x_T$}. The transition state theory offers a framework for sampling such rare events \cite{weinan2006towards,lu2015reactive}. By means of the Doob transformation \cite{doob1957conditional}, the stochastic differential equation \eqref{SDE_model}, starting at $x_0$ and ending at $x_T$, has the same law as a Markovian bridge process \cite{ccetin2016markov} governed by 
\begin{equation}
\begin{split}\label{MarBri}
dY(t)=\left[f(Y(t))+A\frac{\nabla p(x_T,T \mid Y(t),t)}{p(x_T,T \mid Y(t),t)}\right]dt+\varepsilon\sigma d\hat{W}(t),
\end{split}
\end{equation}
where the covariance matrix $A = \varepsilon^2\sigma\sigma^T$, $\hat{W}$ is a Brownian motion in $\mathbb{R}^d$ and the transition probability density function $p(x_T,T \mid x,t)$ satisfies the Kolmogorov backward equation
\begin{equation}
\frac{\partial p\left(x_{T}, T \mid x, t\right)}{\partial t}=-f(x) \cdot \nabla p\left(x_{T}, T \mid x, t\right)-\frac{1}{2} A \triangle p\left(x_{T}, T \mid x, t\right).
\end{equation}
We remark that, with probability one, the samples of this Markovian bridge process \eqref{MarBri} will move from initial point $x_0$ to final point $x_T$ when the time $t$ goes to $T$.

Similar to the sample with the forward committor function \cite{khoo2019solving,lu2015reactive,li2019computing}, we could not characterize the transition probability density function $p(x_T,T \mid x,t)$ analytically in general. Moreover, computing the transition probability density function highly needs the information from \eqref{SDE_model}, since we want to learn the drift function from the observations of transition paths. An alternative way to deal with this problem is to make approximations. To this end, we take $d=k$ and the diffusion coefficient as the identity matrix in \eqref{SDE_model}, i.e. $\sigma=I$.

\textcolor{red}{When the transition time $T$ is short}, the Markovian bridge process \eqref{MarBri} can be approximated \cite{orland2011generating} by the following stochastic differential equation
\begin{equation}\label{MarBri1}
d Z(t)=\left[\frac{x_{T}-Z(t)}{T-t}-\frac{T-t}{4} \nabla g(Z(t))\right]dt +  \varepsilon d \hat{W}_{t},\  t\in [0,T),
\end{equation}
where $g(x)=|f(x)|^2-\varepsilon\nabla\cdot f(x)$. Using this approximation, we can sample the transition paths for \eqref{SDE_model} under the Onsager-Machlup framework.

\textcolor{red}{When $\varepsilon$ tends to zero}, the Markovian bridge process \eqref{MarBri} in one dimension can be approximated \cite{delarue2017ab} by the following stochastic differential equation
\begin{equation}\label{MarBri2}
d Z(t)\!\!=\!\!\left[\frac{x_{T}\!-\!Z(t)}{T\!-\!t}\!-\!\frac{T-t}{2} \int_{0}^{1}(1\!-\!u) \frac{d}{d x}\left[\left(f(x)\right)^{2}\right]\left(x_{T} u\!+\!Z(1-u)\right) d u\right]dt \!+\! \varepsilon d\hat{W}(t),
\end{equation}
for $t\in[0,T)$. Using this approximation, we can sample the transition paths for \eqref{SDE_model} under the Freidlin-Wentzell framework.

These transition paths can be considered as the simulation of the original stochastic differential equations \eqref{SDE_model}. The most probable transition pathways of stochastic differential equations \eqref{SDE_model} and \eqref{MarBri} coincide \cite{huang2021most}. In the next section, we will use the expectation of transition paths as the most probable transition pathway observation data to extract the drift function. 

\subsection{Neural Networks}
Let $h^{L}: \mathbb{R}^{n_0} \rightarrow \mathbb{R}^{n_L}$ be a fully connected neural network with depth $ L$ and width $n_{j}$ in the $j$-th layer. In the $j$-th layer, denote the weight matrix $w^{j} \in \mathbb{R}^{n_{j} \times n_{j-1}}$ and the bias vector $b^{j} \in \mathbb{R}^{n_{j}}$. Let $\theta_{L}:=\left\{w^{j}, b^{j}\right\}_{1 \leq j \leq L}$. Given an activation function $\sigma(\cdot)$, the neural network is defined by
\begin{equation} \label{eqn11}
\begin{aligned}
h^{j}(\mathbf{x}) & =\left\{
             \begin{array}{lr}
          w^{j} \sigma\left(h^{j-1}(\mathbf{x})\right)+b^{j} , \quad \text { for } \quad 2 \leq j \leq L, &  \\

            w^{1} \mathbf{x}+b^{1},\ \  \quad \quad \quad \quad \ \ \text { for } \quad j = 1. &  \\
             \end{array}
\right.
\end{aligned}
\end{equation}
The input is $\mathbf{x} \in \mathbb{R}^{n_{0}}$, and the output of the $j$-th layer is $h^{j}(\mathbf{x}) \in \mathbb{R}^{n_{j}}$.
Popular choices of activation functions include the sigmoid $\left(1 /\left(1+e^{-x}\right)\right)$, the hyperbolic tangent $(\tanh (x))$, and the rectified linear unit $(\max \{x, 0\})$. Note that $h^{L}$ is called a $(L-1)$ hidden layer neural network or a $L$-layer neural network.

Since neural network $h^{L}(\mathbf{x})$ depends on the parameters $\theta_{L}$, we often denote $h^{L}(\mathbf{x})$ by $h^{L}\left(\mathbf{x} ; \theta_{L}\right) .$  Given a neural network structure, we define a neural network function class
\begin{equation}\label{NeuNet}
\mathcal{H}=\left\{h^{L}\left(\cdot ; \theta_{L}\right): \mathbb{R}^{n_0} \mapsto \mathbb{R}^{n_L} \mid \theta_{L}=\left\{\left(w^{j}, b^{j}\right)\right\}_{j=1}^{L}\right\}.
\end{equation}

\section{Methodology}
Under the Onsager-Machlup framework and Freidlin-Wentzell framework, the minimizer of the action functional means the largest probability of the occurrence of a transition. We use neural networks to compute the most probable transition pathway by solving the Euler-Lagrange equation and compare it with the expected transition paths simulated by the Markovian bridge process. Then, we consider the inverse problem to learn the drift function $f$ and most probable transition pathway from the observation data.  

\subsection{Forward problem of the Euler-Lagrange equation}
Consider the following Euler-Lagrangian equation:
\begin{equation}
\begin{split}\label{EL_Forward}
\ddot{z}=g(z,\dot{z}),
\end{split}
\end{equation}
with initial point $z(0)=x_0$ and final point $z(T)=x_T$, where $g$ is related to the drift function $f$, and the diffusion function $\sigma$ of the stochastic dynamical system \eqref{SDE_model}. Here, equation \eqref{EL_Forward} could be the Euler-Lagrangian equation in both the Onsager-Machlup framework \eqref{el_om} and the Freidlin-Wentzell framework \eqref{el_fw}. 

We recall the notations for the function spaces. Denote $U=(0,T)$. For a positive integer $k$, let $C^{k}(U)$ be the set of vector-valued functions whose derivatives of order $\leq k$ are continuous in $U$. Denote $C^k(\overline{U})$ be the set of functions in $C^k(U)$ whose derivatives of order $\leq k$ have continuous extensions to $U$. A function $f$ is H\"{o}lder continuous with exponent $\alpha\ (0<\alpha\leq 1)$ in $U$, if 
\begin{equation}
[f]_{\alpha ; U}=\sup _{x, y \in U, x \neq y} \frac{\|f(x)-f(y)\|}{|x-y|^{\alpha}}<\infty,
\end{equation}
where $\|\cdot\|$ is the Euclidean norm in $\mathbb{R}^d$. Note that if $\alpha=1$, $f$ is a Lipschitz vector-valued function.

For positive integer $k$ and exponent $\alpha$ with $0<\alpha\leq 1$, we define the H\"{o}lder space
\begin{equation}
C^{k,\alpha}(\overline{U})=\{f\in C^{k}(\overline{U})|f^{(k)} \text{ is  H\"{o}lder continuous with exponent $\alpha$}\},
\end{equation}
where $f^{(k)}$ is the derivative of order $k$ for $f$. Equipping with the norm 
\begin{equation}
\|f\|_{k,\alpha}=\sum\limits_{j=1}^k\sup\limits_{x\in U}\|f^{(j)}(x)\|+[f^{(k)}]_{\alpha ; U},
\end{equation}
the H\"{o}der space $C^{k,\alpha}(\overline{U})$ is a Banach space. Denote $[f]_{k,\alpha ; U}=[f^{(k)}]_{\alpha ; U}$ and $[f]_{0,\alpha ; U}=[f]_{\alpha ; U}$.

We will consider the solution of the Euler-Lagrange equation \eqref{EL_Forward} in the H\"{o}lder spaces $C^{k,\alpha}([0,T])$ and use the PINNs method \cite{raissi2019physics,shin2020convergence} to compute the most probable transition pathway. Comparing to solving the partial differential equations, the solution of the Euler-Lagrange equation is a vector-valued function in a time interval $[0,T]$. They proved convergence results for linear partial differential equations. In our case, the convergence results for nonlinear differential equations are proved.

We approximate the solution to the Euler-Lagrange equation \eqref{EL_Forward} from a set of training data, which consist of residual data, initial and final data. Denote the residual data as $\{(t_j,[\ddot{z}-g(\dot{z},z)](t_j))\}_{j=1}^m$, where $t_j\in U=(0,T)$, and the initial and final data $\{(0,x_0),(T,x_T)\}$. We denote the residual input data points as $\mathcal{T}^m=\{t_j\}_{j=1}^m$, where $m$ represents the number of residual training data points in $(0,T)$.

Given a class of neural networks $\mathcal{H}$ defined as \eqref{NeuNet}, we seek to find a neural network $h^{*}$ in $\mathcal{H}$ that minimizes an objective (loss) function. Let $h\in \mathcal{H}$. We consider the expected PINNs loss
\begin{equation}
\begin{split}
\operatorname{Loss}^{\operatorname{PINN}}\left(h,\lambda\right)\! = \!\lambda_r\|[\ddot{h}-g(\dot{h},h)\|_{L^2(U;\mu)}^2\!+\! \frac{\lambda_b}{2}(\|h(0)\!-\!x_0\|^2+\|h(T)\!-\!x_T\|^2),
\end{split}
\end{equation}
where $\|\cdot\|$ is the Euclidean norm in $\mathbb{R}^d$, $\lambda_r$ and $\lambda_b$ are positive constants with $\lambda=(\lambda_r,\lambda_b)$, and $\mu$ is a probability distribution in $U=(0,T)$. 

Suppose the residual input data points $\mathcal{T}^m$ are independently and identically distributed samples from probability distribution $\mu$. Let the empirical probability distribution on $\mathcal{T}^m$ by $\mu^m=\frac{1}{m}\sum_{j=1}^m\delta_{t_j}$. Thus, the empirical PINN loss is given by
\begin{equation}
\begin{split}\label{emp_loss}
\operatorname{Loss}_{m}^{\operatorname{PINN}}\left(h,\lambda\right)\!\!=\!\!\frac{\lambda_r}{m}\sum\limits_{j\!=\!1}^m\|[\ddot{h}\!-\!g(\dot{h},h)](t_j))\|^2\!+\! \frac{\lambda_b}{2}(\|h(0)\!-\!x_0\|^2\!+\!\|h(T)-x_T\|^2).
\end{split}
\end{equation}

We use neural network training to find the solution to the Euler-Lagrange equation \eqref{EL_Forward}. However, sometimes we do not know whether the neural network class $\mathcal{H}$ contains the exact solution $z^{*}$ of the Euler-Lagrange equation. \textcolor{red}{Even if we find a solution in $\mathcal{H}$, there is no guarantee that the minimizer we found matches with the solution $z*$, because there could be numerous local minimizers.}

Suppose that there always exists a unique solution to the Euler-Lagrange equation \eqref{EL_Forward}. Since the expected PINN loss depends on the selection of the training data distribution $\mu$, we assume that the training data distribution $\mu$ satisfies the probabilistic space filling arguments conditions \cite{calder2019consistency,shin2020convergence}. Thus, we make the following assumptions.
\begin{Assu}
Let $\mu$ be probability distributions defined on $U=(0,T)$ and $\rho$ be its probability density with respect to the Lebesgue
measure on $U$.

1. The probabilty density $\rho$ is supported in $U$ and $\operatorname{inf}_{U}\rho>0$.

2. For every $\delta>0$, there exists a partition of $U$, $\{U_j^{\delta}\}_{j=1}^{K_{\delta}}$, that depend on $\delta$ such that for each $j$, there exists a interval $H^{\delta}(t_j))$ of length $\delta$ centered $t_j\in U$, respectively, satisfying $U_j^{\delta}\subset H^{\delta}(t_j)$.

3. There exist positive constants $c$ and $C$, such that for all $\delta>0$, the partitions from the above satisfy $\mu\left(U_{j}^{\epsilon}\right)\geq c\delta$  for all $j$ and for all $t \in U$, $\mu(B_{\delta}(t)\cap U)\leq C\delta$, where $B_{\delta}(t)$ is a closed interval of radius $\delta$ centered at $t\in U$. Here, the constants $c, C$ depend only on $\left(U, \mu\right)$.
\end{Assu}

We state the result that bounds the expected PINN loss in terms of the empirical PINN loss.
\begin{thm}
Suppose Assumption 1 holds. Assume that $g$ is Lipschitz for the variables $z$ and $\dot{z}$ with a Lipschitz constant $C_L$. Let $\mathcal{T}^m$ are independently and identically distributed samples from probability distribution $\mu$. For some $\alpha$ with $0<\alpha\leq 1$, let $h$ satisfy 
\begin{equation}
R(h)=[h]_{0,\alpha;U}+[h]_{1,\alpha;U}+[h]_{2,\alpha;U}<\infty. 
\end{equation}
Then, with probability $1-\sqrt{m}(1-1/\sqrt{m})^m$, at least, we have
\begin{equation}
\begin{split}
\operatorname{Loss}^{\operatorname{PINN}}\left(h;\lambda\right)\leq \frac{C}{c}m^{\frac{1}{2}} \cdot\left(\operatorname{Loss}_m^{\operatorname{PINN}}\left(h;\lambda\right) + \lambda_m^R R(h)\right),
\end{split}
\end{equation}
where $\lambda_m^R=\frac{3(C_L^2+1)}{C\cdot c^{{2\alpha-1}}}\cdot m^{-\alpha-\frac{1}{2}}$.
\end{thm}
\begin{proof}
The detailed proof is in Appendix A.
\end{proof}
\begin{rem}
(i) The Assumption 1 can be satisfied in many practical cases, for example, the uniform probability distributions on $U$. Also note that the functional $R(h)$ is independent of the samples and $\lambda_m^R\rightarrow0$, as $m\rightarrow\infty$. Thus, when $m\rightarrow\infty$, the influence of the functional term $R(h)$ will fade away. We call $R(h)$ the regularization functional, and denote the H\"{o}lder regularized loss as
\begin{equation}
\begin{split}\label{RegLos}
\operatorname{Loss}_m^{\operatorname{PINN}}\left(h;\lambda,\lambda_m^R\right)= \operatorname{Loss}_m^{\operatorname{PINN}}\left(h;\lambda\right) + \lambda_m^R R(h).
\end{split}
\end{equation}
(ii) Let the equation \eqref{EL_Forward} \textcolor{red}{coincide with} the Euler-Lagrange equation \eqref{EL_eqn}. If $A=\sigma\sigma^T$ is no-degenerate, and $f,\nabla f,A,\nabla A$ are bounded Lipschitz, then the function $g$ is Lipschitz for the variables $z$ and $\dot{z}$ in the Freidlin-Wentzell case. If adding the condition that the gradient of $\nabla f$ is bounded Lipschitz, then $g$ is Lipschitz for the variables $z$ and $\dot{z}$ in the Onsager-Machlup case.
\end{rem}

Next, we will show the convergence result of the empirical PINNs loss under the following assumptions.
\begin{Assu}Denote $U=(0,T)$. Let $k$ be the highest order of the derivative shown in the Euler-Lagrange equation \eqref{EL_Forward} and $\alpha\in(0,1]$. Suppose that

1. for each $m$, let $\mathcal{H}$ be a class of neural networks in $C^{k,\alpha}(\overline{U})$ such that for all $h\in \mathcal{H}$, $\ddot{h},\dot{h},h\in C^{0,\alpha}(U)$;

2. for each $m$, $\mathcal{H}$ contains a network $z_m^{*}$ satisfying $\operatorname{Loss}_m^{\operatorname{PINN}}(z_m^{*};\lambda)\!=\!\mathcal{O}(m^{-\!\alpha\!-\!\frac{1}{2}})$;

3. and we have,
\begin{equation}
R^{*}=\sup\limits_{m}R(z_m^{*})=\sup\limits_{m}\left([z_m^{*}]_{0,\alpha;U}+[z_m^{*}]_{1,\alpha;U}+[z_m^{*}]_{2,\alpha;U}\right)<\infty. 
\end{equation}
\end{Assu}

Now we state the following theorem.
\begin{thm}
Suppose Assumption 1 and Assumption 2 hold. Assume that $g$ is Lipschitz for the variables $z$ and $\dot{z}$ with a Lipschitz constant $C_L$. Let $\mathcal{T}^m$ be independently and identically distributed samples from probability distribution $\mu$. For some $\alpha$ with $0<\alpha\leq 1$, let $h_m$ be a minimizer of the H\"{o}lder regularized loss $\operatorname{Loss}_m^{\operatorname{PINN}}\left(h_m;\lambda,\lambda_m^R\right)$ defined in \eqref{RegLos}. Then, with probability $1-\sqrt{m}(1-1/\sqrt{m})^m$, at least, we have 
\begin{equation}
\begin{split}
\operatorname{Loss}^{\operatorname{PINN}}\left(h_m;\lambda\right)=\mathcal{O}(m^{-\alpha}).
\end{split}
\end{equation}
\end{thm}
\begin{proof}
The detailed proof is in Appendix B.
\end{proof}

\subsection{Inverse problem of the parameterized Euler-Lagrange equation}\label{inv1}
Consider the following parameterized Euler-Lagrange equation 
\begin{equation}
\begin{split}\label{EL_Par}
\ddot{z}=g(z,\dot{z},\beta),
\end{split}
\end{equation}
with $z(0)=x_0, z(T)=x_T$, where the parameter $\beta$ comes from the drift function $f(z,\beta_1)$ and diffusion function $\sigma$ of the stochastic dynamical system \eqref{SDE_model} and $\beta=(\beta_1,\sigma)$. Here $g$ is related to the drift function $f$ and the diffusion function $\sigma$.

In this subsection, we learn the parameters of the drift function for the stochastic dynamical system \eqref{SDE_model} under both the Onsager-Machlup framework and the Freidlin-Wentzell framework through the parameterized Euler-Lagrange equation and observation data. The observation data is calculated by the Markovian bridge process \eqref{MarBri1} for the Onsager-Machlup framework and \eqref{MarBri2} for the Freidlin-Wentzell framework, corresponding to the stochastic differential equation. We simulate 10 transition paths for each framework, and then take the expectation as an approximation of the most probable transition pathway. For a comparison, we also use the most probable transition pathway, computed by the neural network, as the observation data.  

However, in high dimensions, there is no simple approximation like \eqref{MarBri1} and \eqref{MarBri2} in one dimension to simplify the Markovian bridge process \eqref{MarBri}. So we would have to compute the committor function or the Kolmogorov backward equation, which is not ideal. Instead, we do a Gaussian perturbation of the most probable transition pathway, computed by the neural network as the observation data to approximate the transition paths.

Denote the observation data $\{z(t_i)\}_{i=1}^{N}$ with $z(t_0)=x_0$ and $z(t_N)=x_T$. Let $h\in\mathcal{H}$. We define the observation data loss as
\begin{align}
\operatorname{Loss}_{ob}&=\frac{1}{N} \sum_{j=1}^{N}(h(t_i)- z(t_i))^2 .\
\end{align}
Thus our objective (loss) function is 
\begin{equation}
\begin{split}\label{Los_Fun}
\operatorname{Loss}^{\operatorname{PINN}}\left(h,\beta;\lambda\right)=&\|[\ddot{h}-g(h,\dot{h},\beta)\|_{L^2(U;\mu)}^2+ \frac{\lambda_d}{N} \sum_{j=1}^{N}\|h(t_j)- z(t_j)\|^2,
\end{split}
\end{equation}
where $\lambda_d$ is a positive constant, which is used to balance the PINN loss and observation data loss. Then we use Adam optimizer to train the loss function to obtain the optimal neural network and parameters $\beta$ in the stochastic differential equation.

\subsection{Inverse problem of the non-parameterized Euler-Lagrange equation}
In subsection \ref{inv1}, we parameterized the drift function and learned the parameters and the most probable transition pathway from sparse observation data. In this subsection, we consider the unknown drift case. We use a neural network to approximate the drift function and discretize the Euler-Lagrange equation to obtain the loss function. Thus, let us suppose the Euler-Lagrange equation:
\begin{equation}
\begin{split}
\ddot{z}=g(z,\dot{z},f(z)),
\end{split}
\end{equation}
where $f$ is the drift function of the stochastic differential equation. Let the observation data be $\{z(t_i)\}_{i=1}^{N}$, where $t_i=i\tau$ and $\tau=\frac{T}{N}$.
The discretization scheme for the Euler-Lagrange equation is
\begin{equation}
\begin{split}
\frac{z(t_{i+1})-2z(t_i)+z(t_{i-1})}{\tau^2}=g(z(t_i), \frac{z(t_{i+1})-z(t_{i-1})}{2\tau},f(z_i)).
\end{split}
\end{equation}
We construct a fully connected neural network $f_{NN}$ to approximate the drift term $f$ with input $z$. Therefore, the loss function of non-parameterized Euler-Lagrange equation is defined as
\begin{equation}
\begin{split}
\operatorname{Loss}_{ode}\!\!=\!\!\frac{1}{N\!\!-\!\!2}\sum_{i=2}^{N\!-\!1}\bigg(\frac{z(t_{i+1})\!-\!2z(t_i)\!+\!z(t_{i-1})}{\tau^2}\!-\!g(z(t_i),\frac{z(t_{i+1})-z(t_{i-1})}{2\tau},f_{NN}(z_i)) \bigg)^2.
\end{split}
\end{equation} 
The goal is to learn the drift function. But the observation is only one path, which is difficult to recover the stochastic differential equation well. So we also add some observation data of the drift function, $\{f(z_i)\}_{i=1}^{N_d}$. We use mean square error to measure this loss
 \begin{equation}
\begin{split}
\operatorname{Loss}_{drift}=\frac{1}{N_d}\sum_{i=1}^{N_d}\big(f_{NN}(z_i)-f(z_i) \big)^2.
\end{split}
\end{equation} 
Therefore, the total loss of non-parameterized Euler-Lagrange equation is
 \begin{equation}\label{loss_unknow}
\begin{split}
\operatorname{Loss}=\operatorname{Loss}_{ode}+\gamma_1 \operatorname{Loss}_{drift}+\gamma_2 ||w_d||_{2}^2,
\end{split}
\end{equation} 
where $||w_d||_{2}$ is the $L_2$ regularization with the weights $w_d$ in the neural network, and $\gamma_1$ and $\gamma_2$ are positive constant parameters to balance the three loss terms. 
Note that if $\gamma_1=0$, no observation data of the drift function is considered, and if $\gamma_2=0$, no regularization of the weights is considered.

\section{Numerical experiments}
In this section, we conduct the numerical experiments for several examples in both the Freidlin-Wentzell framework and the Onsager-Machlup framework. For these examples, we compute the most probable transition pathway in two ways. One is averaging the samples of the transition paths generated from the Markovian bridge process, and the other is using the PINN method to solve the Euler-Lagrange equation. Moreover, we recover stochastic differential equations from the observation data in both parameter and non-parameter cases.  

In our experiments, the Markovian bridge process data are computed by the Markovian bridge process \eqref{MarBri1} and \eqref{MarBri2}. In the Freidlin-Wentzell framework, we simulate 10 transition paths with noise intensity $\sigma =0.0001$, using the approximate Markovian bridge process \eqref{MarBri2}, and average these 10 transition paths to approximate the most probable transition pathway. In the Onsager-Machlup framework, we chose the approximate Markovian bridge process \eqref{MarBri1} with the noise intensity $\sigma =0.1$. All the neural networks have 4 hidden layers and 20 neurons per layer, with $\operatorname{tanh}$ activation function. The weights are initialized with truncated normal distributions and the biases are initialized as zero.
The Adam optimizer with a learning rate of $10^{-4}$ is used to train the loss function. In one dimension, the number of residual points for evaluating the Euler-Lagrange equation is $m = 1001$, while in two dimensions the number is $m=501$. 

\subsection{Stochastic double-well system}
Consider the following stochastic double-well system
\begin{equation}
\begin{split}
dX(t)=(X(t)-X(t)^3)dt+\sigma dW(t), t\in[0,T].
\end{split}
\end{equation}

There are two stable states $x_1=-1$, $x_2=1$ and one unstable state $x_0=0$ for the corresponding deterministic system. We call $x_1$ and $x_2$ are the metastable states in the stochastic double-well system and consider transition phenomena between these two metastable states. 

\noindent \textbf{Freidlin-Wentzell framework}

Under the Freidlin-Wentzell framework, the most probable transition pathway satisfies the Euler-Lagrange equation
\begin{equation}
\begin{split}\label{EL_DW}
\ddot{z}=(1-3z^2)(z-z^3),
\end{split}
\end{equation}
with initial point $z(0)=-1$ and final point $z(T)=1$.

We compute the most probable transition pathway between $z(0)=-1$ and $z(T)=1$ in two different ways for different transition times as shown in Fig.\ref{dw_mptp}(a). The red curves are computed by the physics-informed neural network to the Euler-Lagrange equation \eqref{EL_DW}, through optimising the loss function (3.3) with $\lambda_r=1$ and $\lambda_b=1$.
The green curves are obtained from the Markovian bridge process of transition paths. As shown in Fig.\ref{dw_mptp}(a), these two methods match for small transition times. More precisely, when the transition time is less than 5, the most probable transition pathway averaged of the transition paths is a good approximation of the most probable transition pathway.

\begin{figure}
\begin{minipage}[]{0.48 \textwidth}
 \leftline{~~~~~~~\tiny\textbf{(a)}}
\centerline{\includegraphics[width=5cm,height=5cm]{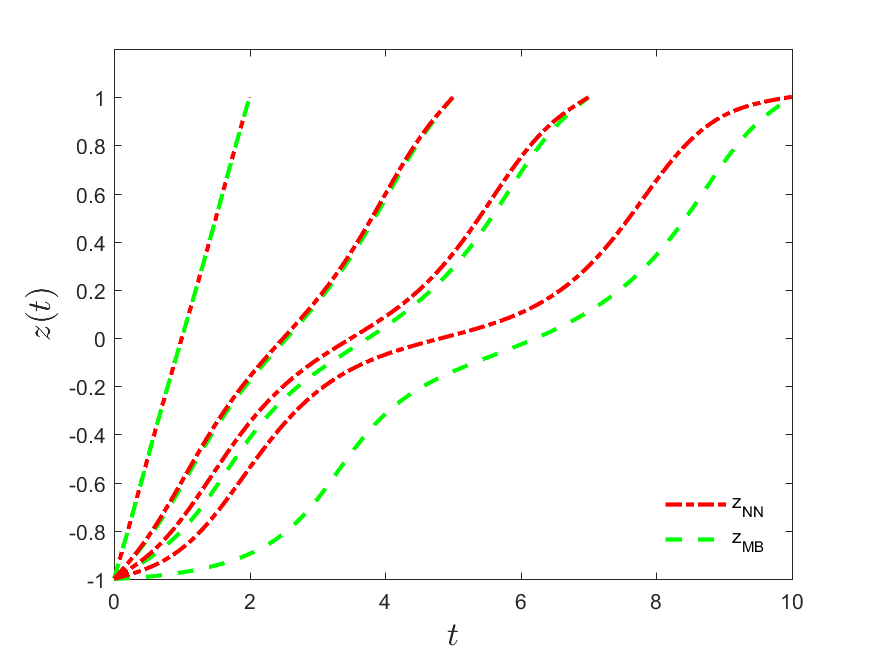}}
\end{minipage}
\hfill
\begin{minipage}[]{0.48 \textwidth}
 \leftline{~~~~~~~\tiny\textbf{(b)}}
\centerline{\includegraphics[width=5cm,height=5cm]{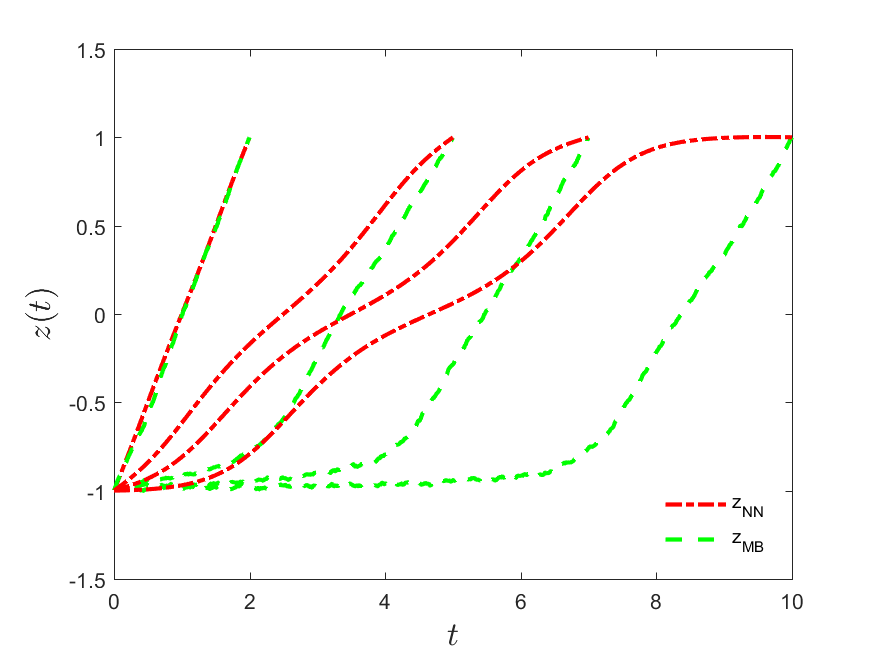}}
\end{minipage}
\caption{\textbf{The most probable transition pathways of the stochastic double-well system for different transition times $T=2, 5, 7, 10$}. (a) Freidlin-Wentzell framework ($\sigma \ll 1$); (b) Onsager-Machlup framework ($\sigma = 0.1$). The red curves are computed by the neural network for the Euler-Lagrange equation. The green curves are computed through the transition paths generated by the Markovian bridge process.}
\label{dw_mptp}
\end{figure}

We also extract the parameters of the drift function ($f(x )=\lambda_1x+\lambda_2x^3$) from the observation data through optimising the loss function (3.11) with $\lambda_d=1$. 
The observation data are shown in Fig.\ref{dw_mptp}(a) and we uniformly choose the number of the observation data $N=21$. We present the parameter evolution predictions as the iteration of the optimizer progresses in Fig.\ref{1D_dw_fw_21}(a). The black curves are true parameters with $\lambda_1=1$ and $\lambda_2=-1$. The green curves are learned by Markovian bridge process observation data, and the red curves are computed by neural network observation data. The learned $\lambda_1$ and $\lambda_2$ are presented in Fig.\ref{1D_dw_fw_21}(a1)-(a4) for different transition times $T=2,5,7,10$. The results show that we can learn the parameters well with an error rate of less than $5\%$ in all cases. Moreover, the convergence is faster for larger transition time $T$.
We present the learned drift function in Fig.\ref{1D_dw_fw_21}(b1)-(b4). We see that the learned drift function matches perfectly with the true drift function. Even for transition time $T=10$, the most probable transition pathways in these two cases vary so much. These results show that a longer transition time gives us more information about the system. 

\begin{figure}
\begin{minipage}[]{0.2 \textwidth}
 \leftline{~~~~~~~\tiny\textbf{(a1)}}
\centerline{\includegraphics[width=3.4cm,height=3.4cm]{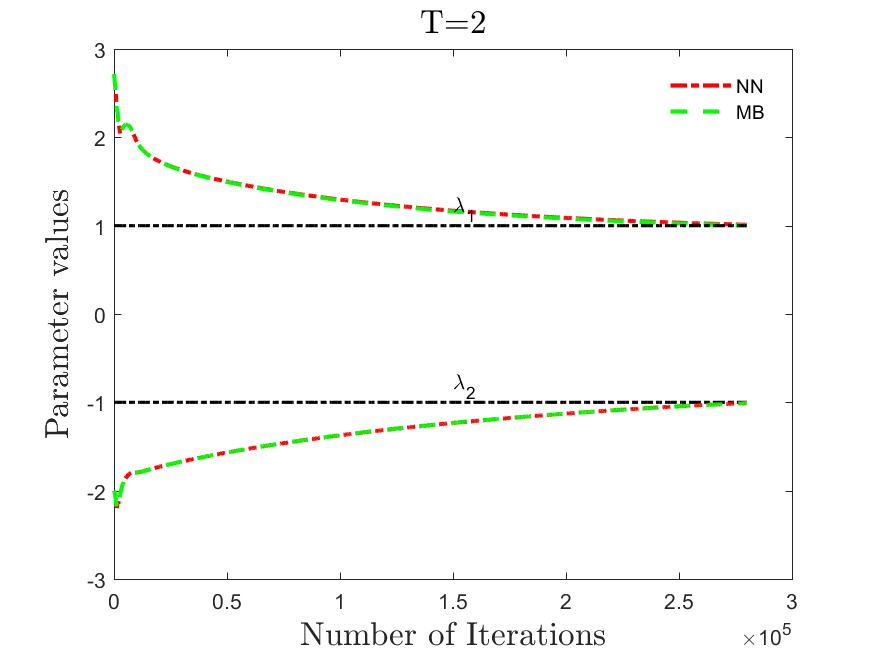}}
\end{minipage}
\hfill
\begin{minipage}[]{0.2 \textwidth}
 \leftline{~~~~~~~\tiny\textbf{(a2)}}
\centerline{\includegraphics[width=3.4cm,height=3.4cm]{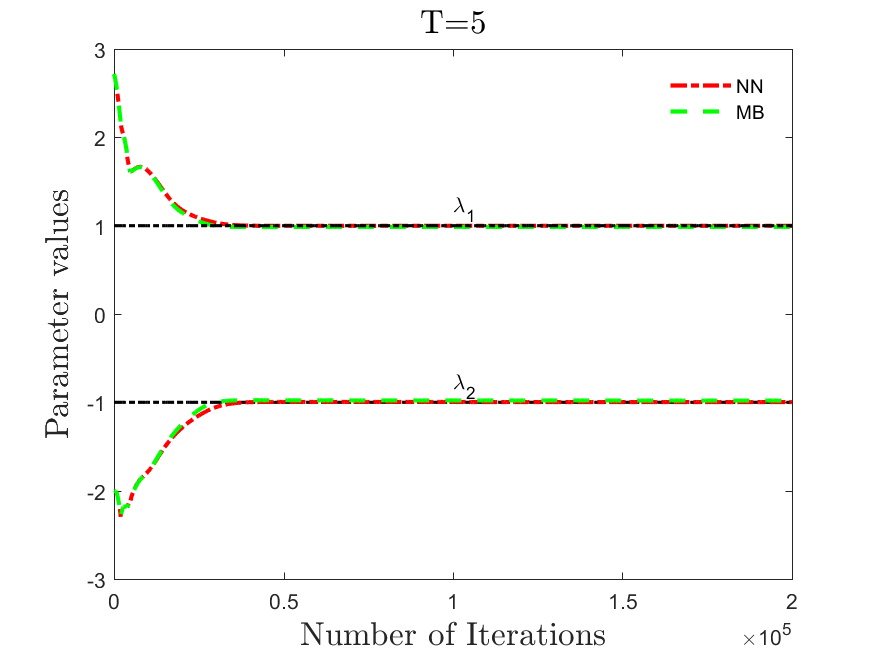}}
\end{minipage}
\hfill
\begin{minipage}[]{0.2 \textwidth}
 \leftline{~~~~~~~\tiny\textbf{(a3)}}
\centerline{\includegraphics[width=3.4cm,height=3.4cm]{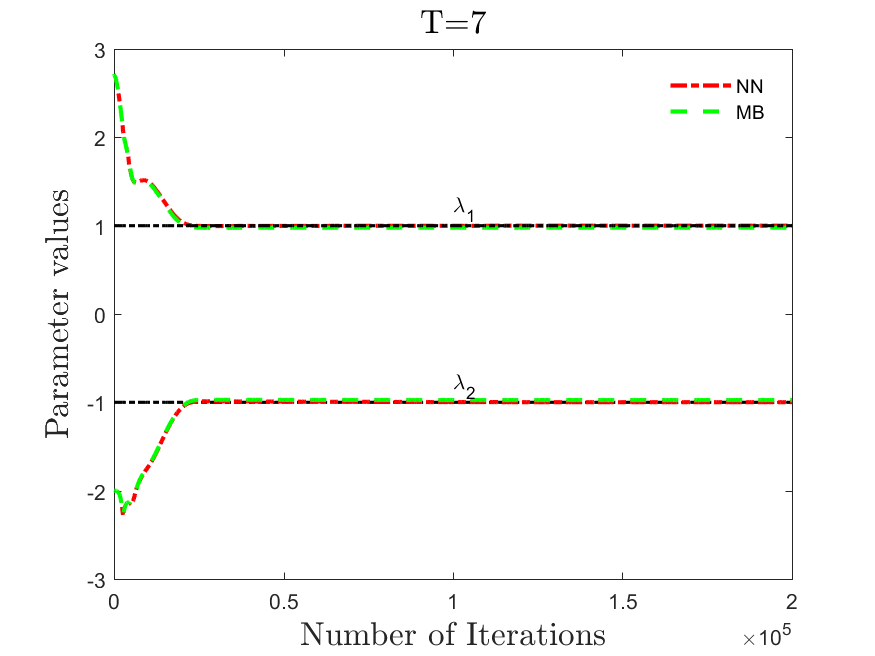}}
\end{minipage}
\hfill
\begin{minipage}[]{0.2 \textwidth}
 \leftline{~~~~~~~\tiny\textbf{(a4)}}
\centerline{\includegraphics[width=3.4cm,height=3.4cm]{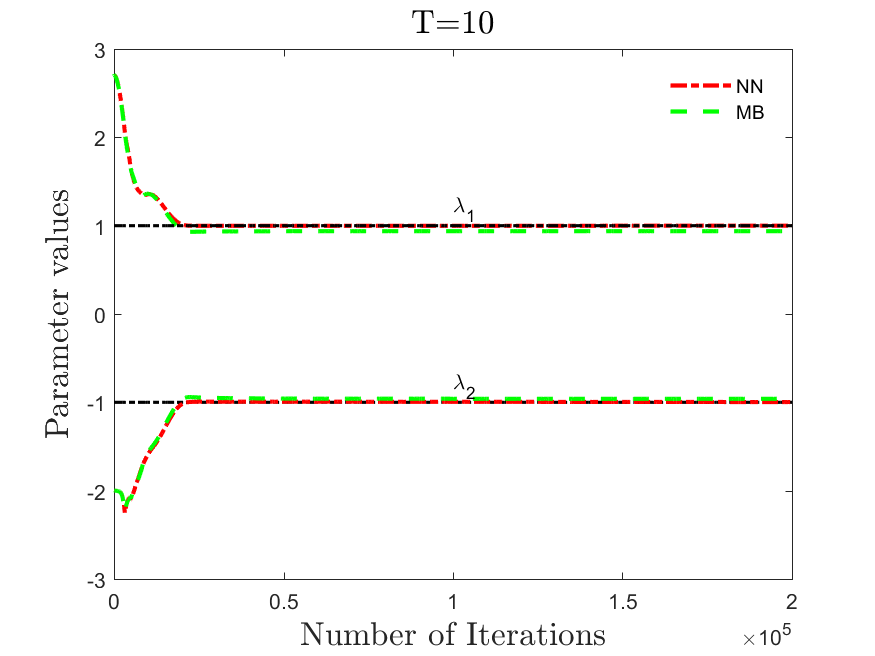}}
\end{minipage}
\hfill
\begin{minipage}[]{0.2 \textwidth}
 \leftline{~~~~~~~\tiny\textbf{(b1)}}
\centerline{\includegraphics[width=3.4cm,height=3.4cm]{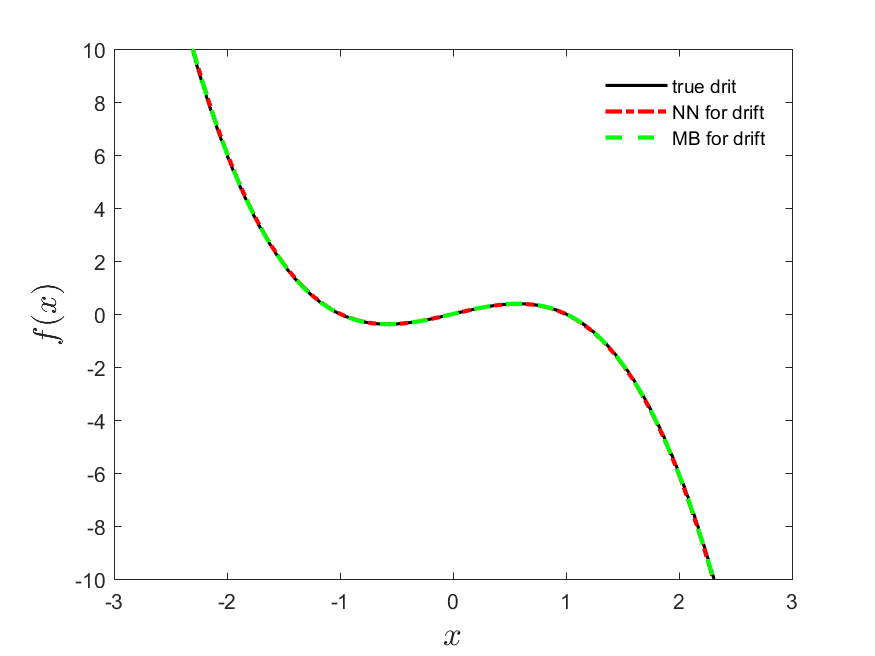}}
\end{minipage}
\hfill
\begin{minipage}[]{0.2 \textwidth}
 \leftline{~~~~~~~\tiny\textbf{(b2)}}
\centerline{\includegraphics[width=3.4cm,height=3.4cm]{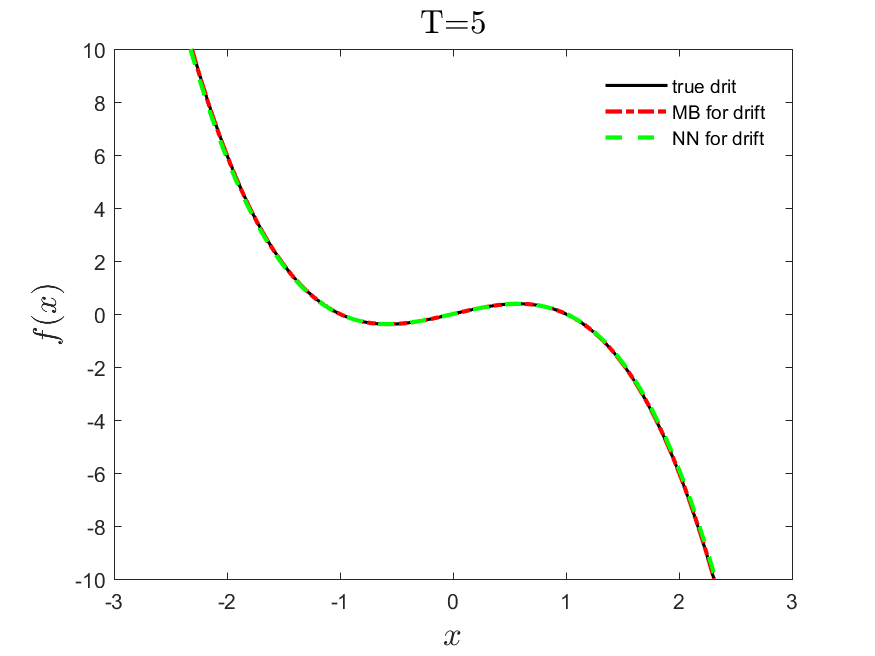}}
\end{minipage}
\hfill
\begin{minipage}[]{0.2 \textwidth}
 \leftline{~~~~~~~\tiny\textbf{(b3)}}
\centerline{\includegraphics[width=3.4cm,height=3.4cm]{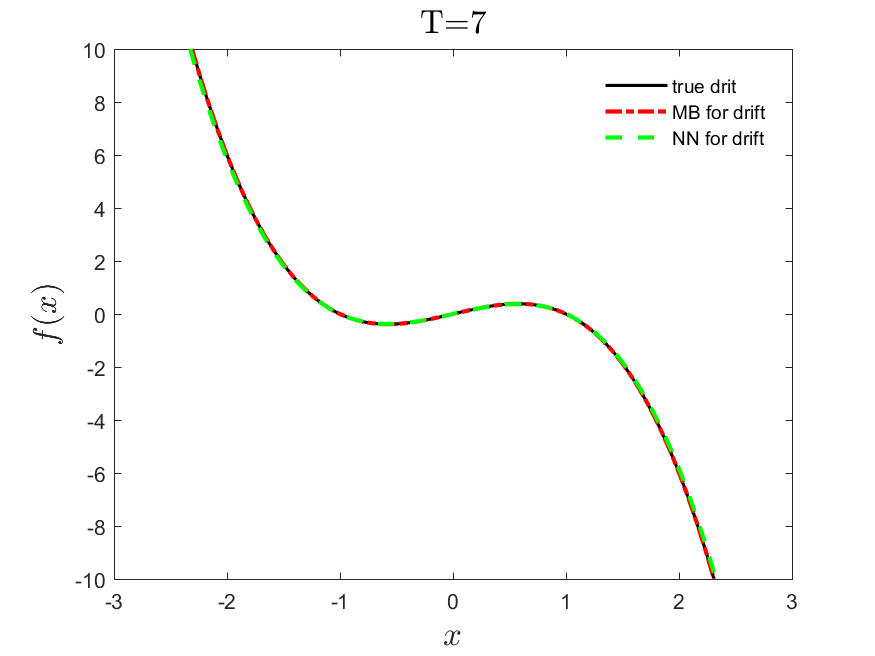}}
\end{minipage}
\hfill
\begin{minipage}[]{0.2 \textwidth}
 \leftline{~~~~~~~\tiny\textbf{(b4)}}
\centerline{\includegraphics[width=3.4cm,height=3.4cm]{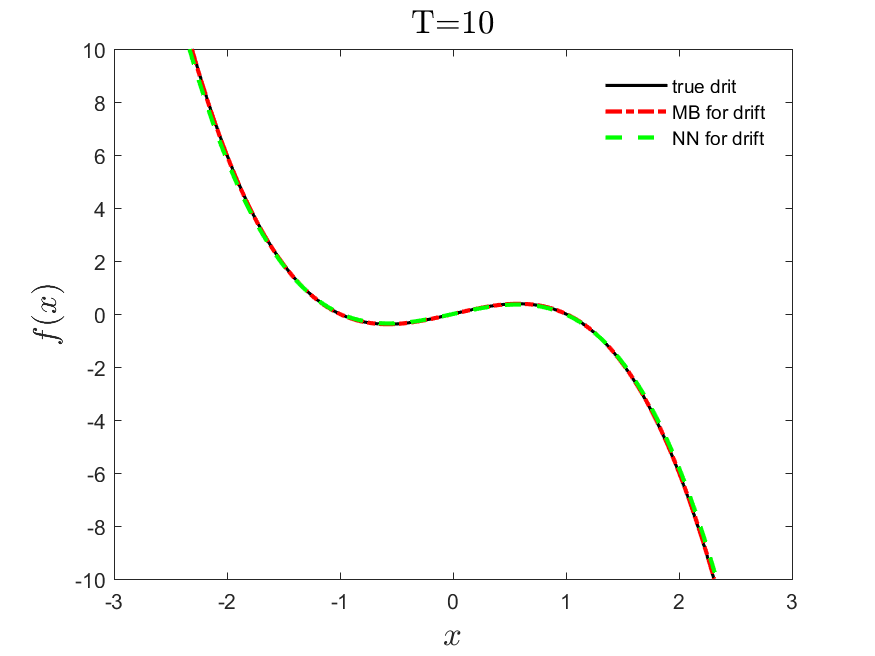}}
\end{minipage}
\caption{\textbf{Parametric estimation in Freidlin-Wentzell case ($\sigma\ll1$) - stochastic double-well system:}  (a) Parameter evolution as the iteration of optimizer progresses for different transition times; (b) learned drift functions for different transition times. Black curves: true parameters and true drift; green curves: Markovian bridge process observation data; red curves: neural network observation data.}
\label{1D_dw_fw_21}
\end{figure}

\noindent \textbf{Onsager-Machlup framework}

The Onsager-Machlup framework is similar to the Freidlin-Wentzell framework, where the Euler-Lagrange equation is
\begin{equation}
\begin{split}\label{EL_dw_om}
\ddot{z}=(1-3z^2)(z-z^3)-3\sigma^2z,
\end{split}
\end{equation}
with two boundary points $z(0)=-1$ and $z(T)=1$. Here, we take the noise intensity $\sigma=0.1$.

We compute the most probable transition pathway between $z(0)=-1$ and $z(T)=1$ for transition times $T=2,5$ as shown in Fig.\ref{dw_mptp}(b). The red curves are computed by the neural network to the Euler-Lagrange equation \eqref{EL_dw_om}, through optimising the loss function (3.3) with $\lambda_r=1$ and $\lambda_b=1$. The green curves are obtained from the transition paths of the Markovian bridge process (2.12). As shown in Fig. \ref{dw_mptp}(b), for small transition times, these two ways of computing the most transition pathways coincide well. 

\begin{figure}
\begin{minipage}[]{0.3 \textwidth}
 \leftline{~~~~~~~\tiny\textbf{(a1)}}
\centerline{\includegraphics[width=4.4cm,height=4.4cm]{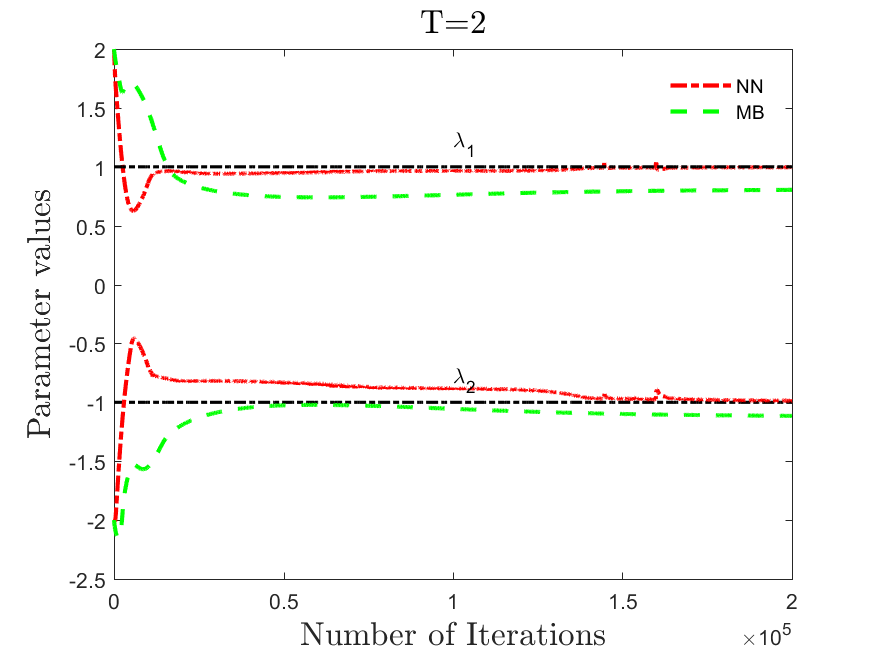}}
\end{minipage}
\hfill
\begin{minipage}[]{0.3 \textwidth}
 \leftline{~~~~~~~\tiny\textbf{(a2)}}
\centerline{\includegraphics[width=4.4cm,height=4.4cm]{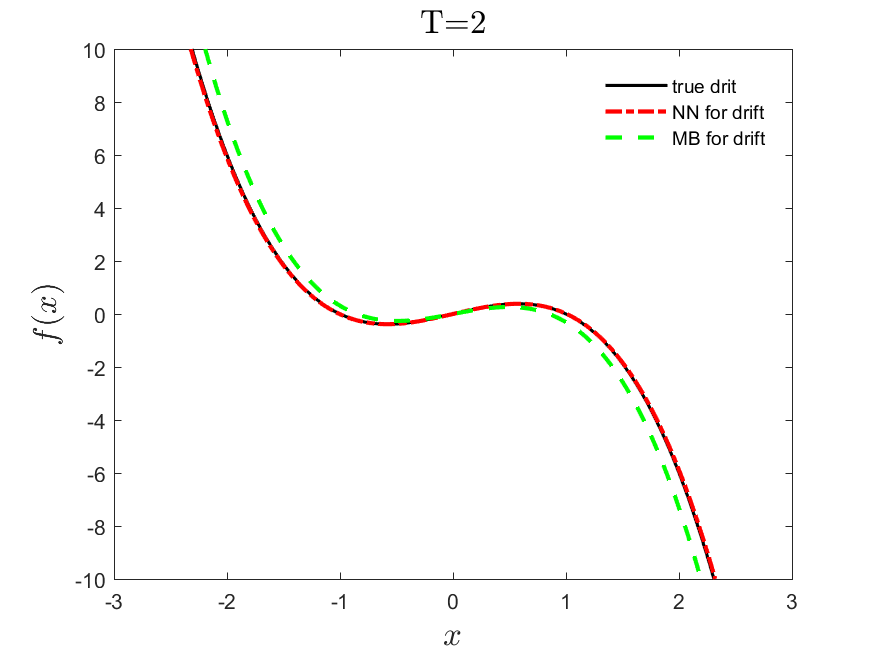}}
\end{minipage}
\hfill
\begin{minipage}[]{0.3 \textwidth}
 \leftline{~~~~~~~\tiny\textbf{(a3)}}
\centerline{\includegraphics[width=4.4cm,height=4.4cm]{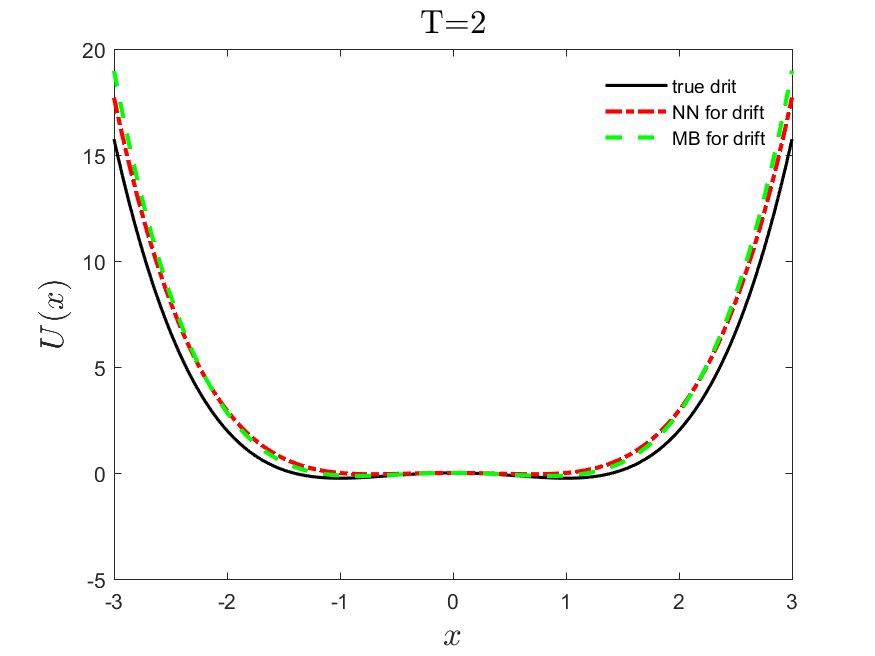}}
\end{minipage}
\begin{minipage}[]{0.3 \textwidth}
 \leftline{~~~~~~~\tiny\textbf{(b1)}}
\centerline{\includegraphics[width=4.4cm,height=4.4cm]{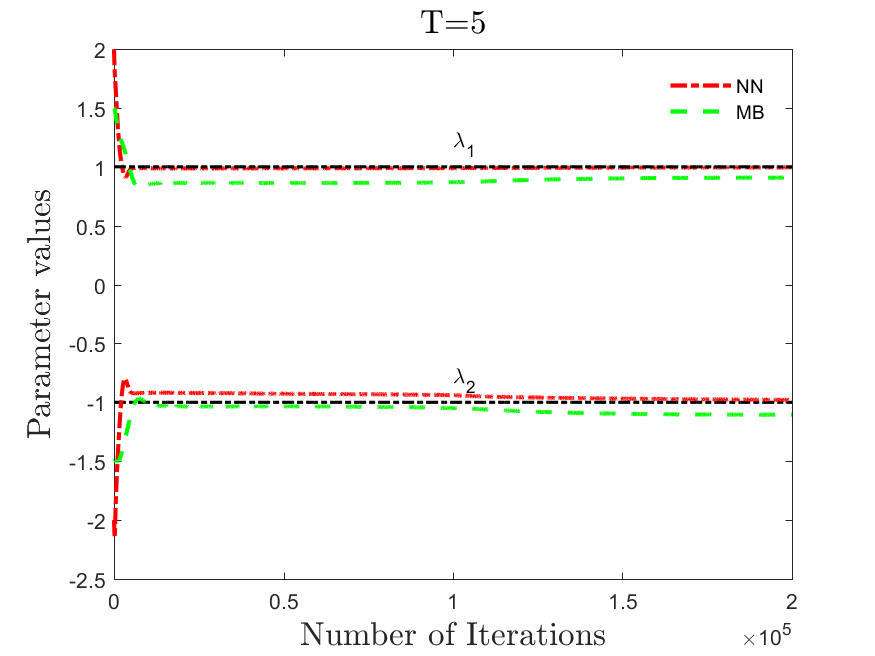}}
\end{minipage}
\hfill
\begin{minipage}[]{0.3 \textwidth}
 \leftline{~~~~~~~\tiny\textbf{(b2)}}
\centerline{\includegraphics[width=4.4cm,height=4.4cm]{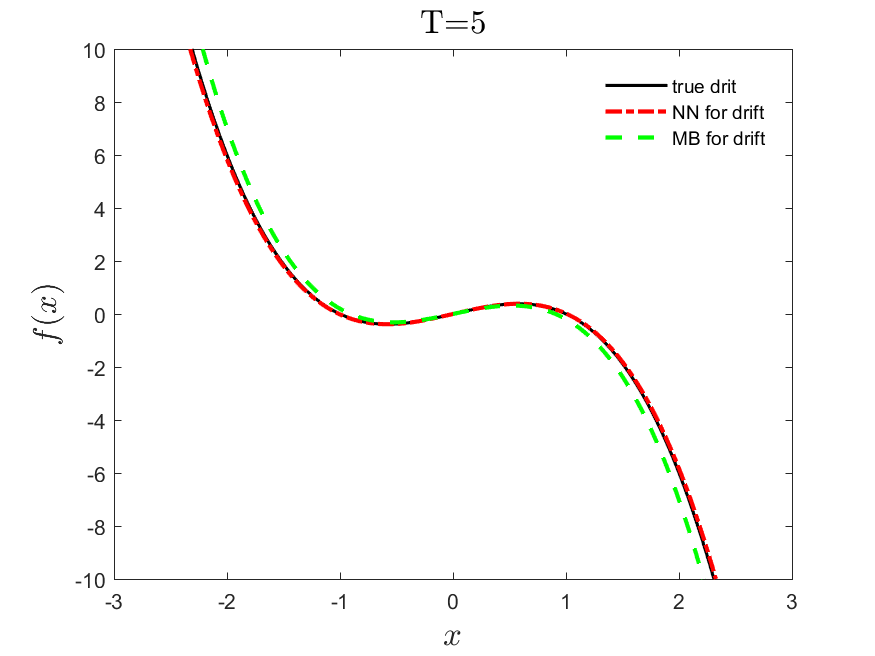}}
\end{minipage}
\hfill
\begin{minipage}[]{0.3 \textwidth}
 \leftline{~~~~~~~\tiny\textbf{(b3)}}
\centerline{\includegraphics[width=4.4cm,height=4.4cm]{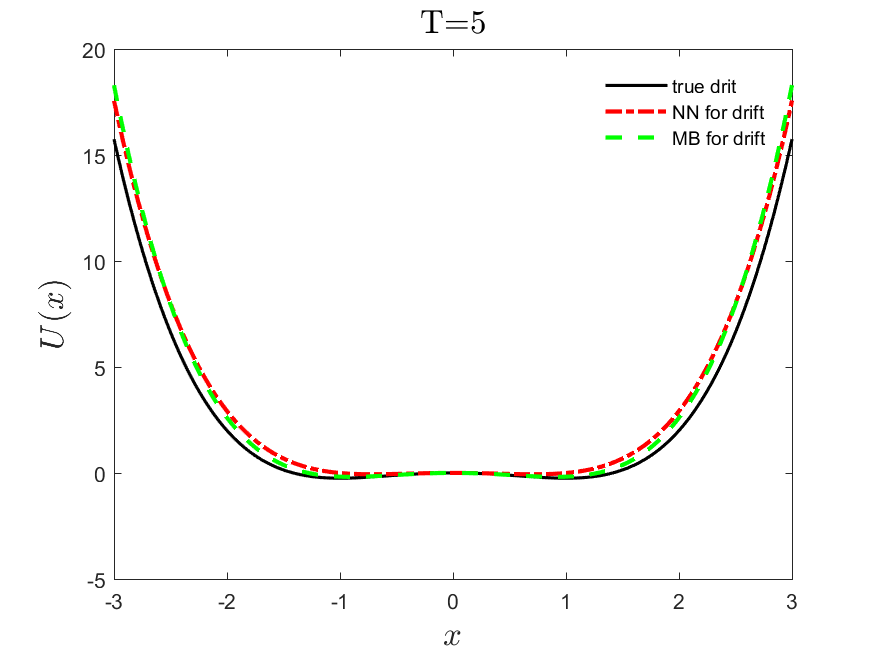}}
\end{minipage}
\caption{\textbf{Parametric estimation in Onsager-Machlup case ($\sigma=0.1$) - stochastic double-well system:} (a) Transition time T=2; (b) Transition time T=5. Left: learned parameter $\lambda_1$ and $\lambda_2$; middle: learned drift function; right: learned potential function. Black lines: true parameters, true drift and true potential; green lines: Markovian bridge process observation data; red lines: neural network observation data.}
\label{1D_dw_om_21}
\end{figure}

Moreover, we recover the parameters of drift function ($f(x )=\lambda_1x+\lambda_2x^3$) from the observation data through optimising the loss function (3.11) with $\lambda_d=10^3$ and learning rate $10^{-3}$. \textcolor{red}{These parameters will have an effect on the behavior of stochastic dynamical systems. As a result, it is critical to learn them.}.
The observation data is shown in Fig.\ref{dw_mptp}(b) and we uniformly choose the number of the observation data $N=21$. Fig.\ref{1D_dw_om_21} shows learned parameters, drift function, and potential function for transition times $T=2$ and $T=5$. Fig.\ref{1D_dw_om_21}(a1) and (b1) present the parameter evolution predictions as the iteration of the optimizer progresses. The black curves are true parameters with $\lambda_1=1$ and $\lambda_2=-1$. Red and green curves are computed by the neural network and the Markovian bridge process observation data, respectively.
The learned drift function and learned potential function are shown in Fig.\ref{1D_dw_om_21}(a2)-(a3) and (b2)-(b3). We could see that the physics-informed neural network could effectively recover the dynamical structures for both the neural network observation data and the Markovian bridge process observation data. Comparing to the transition time $T=2$, the learning results are better for $T=5$.
For $T=2$, we could not learn the parameters well from the Markovian bridge process observation data, although the data is very similar to the neural network data. The reason is that the data for $T=2$ is a straight line, which gives less information about the system and makes it more sensitive to learning the parameters.

\subsection{Stochastic gene regulation model}
Consider the following gene regulation model \cite{smolen1998frequency, wang2018likelihood}
\begin{equation}
\begin{split}\label{tfa_pot}
dX(t)=(\frac{k_fX(t)^2}{X(t)^2+K_d}-k_dX(t)+R_{bas})dt+\sigma dW(t),
\end{split}
\end{equation}
where $\sigma$ is the noise intensity. It is assumed that the transcription rate saturates with the transcription factor activator dimer concentration to a maximal rate $k_f$, transcription factor activator degrades with first-order kinetics with the rate $k_d$, and the transcription factor activator dimer dissociates from specific responsive elements with the constant $K_D$. The basal rate of the synthesis of the activator is $R_{bas}$.

This is a gradient system 
\begin{equation}
\begin{split}
dX(t)=-U'(X(t))dt+\sigma dW(t),
\end{split}
\end{equation}
with the potential 
$U(x)=k_f\sqrt{K_D}\arctan{\frac{x}{\sqrt{K_D}}}+\frac{k_d}{2}x^2-(R_{bas}+k_f)x.$
We choose proper parameters $K_d=10$, $k_f=6$, $k_d=1$ and $R_{bas}=0.6$ in this genetic regulatory system on the basis of genetic significance. Thus, there are two stable states 0.62685 and 4.28343, and a unstable state 1.48971 in the corresponding deterministic system. 

\noindent \textbf{Freidlin-Wentzell framework}

Under the Freidlin-Wentzell framework, the most probable transition pathway satisfies the Euler-Lagrange equation
\begin{equation}
\begin{split}\label{EL_tfa}
\ddot{z}=(\frac{k_fz^2}{z^2+K_d}-k_dz+R_{bas})(\frac{2k_fK_dz}{(z^2+K_d)^2}-k_d),
\end{split}
\end{equation}
with two boundary points $z(0)=0.62685$ and $z(T)=4.28343$.

We compute the most probable transition pathway for different transition times as shown in Fig.\ref{tfa_mptp}(a). The red curves are computed by the neural network to the Euler-Lagrange equation \eqref{EL_tfa}, through optimizing the loss function \eqref{emp_loss} with $\lambda_r=1$ and $\lambda_b=1$. 
The green curves are obtained from the Markovian bridge process of transition paths. As shown in Fig.\ref{tfa_mptp}(a), for different transition times, they coincide very well.

\begin{figure}
\begin{minipage}[]{0.48 \textwidth}
 \leftline{~~~~~~~\tiny\textbf{(a)}}
\centerline{\includegraphics[width=5cm,height=5cm]{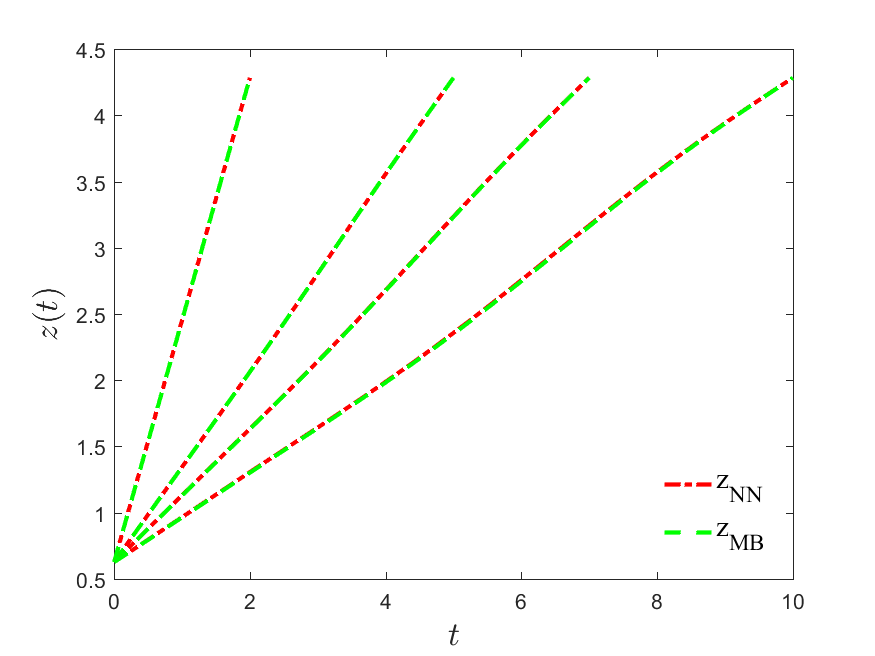}}
\end{minipage}
\hfill
\begin{minipage}[]{0.48 \textwidth}
 \leftline{~~~~~~~\tiny\textbf{(b)}}
\centerline{\includegraphics[width=5cm,height=5cm]{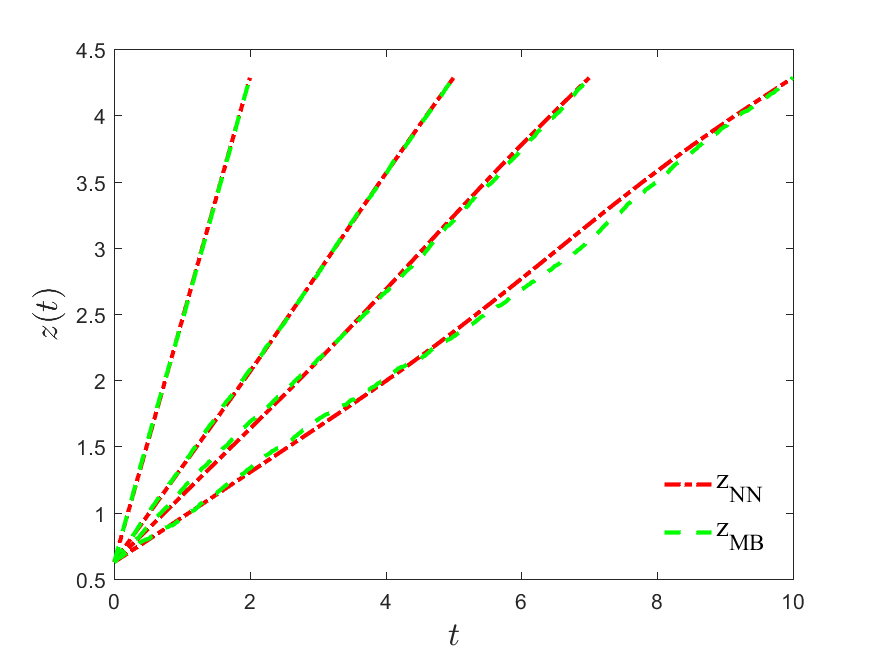}}
\end{minipage}
\caption{\textbf{The most probable transition pathway of the stochastic gene regulation model for different transition times $T=2, 5, 7, 10$}. (a) Freidlin-Wentzell framework ($\sigma\ll1$); (b) Onsager-Machlup framework ($\sigma=0.1$). Red curves are computed by the neural network for the Euler-Lagrange equation. Green curves are computed through the transition paths generated by the Markovian bridge process.}
\label{tfa_mptp}
\end{figure}

\textcolor{red}{
The maximal rate $k_f$, the degradation rate $k_d$, the synthesis rate $R_{bas}$, and the dissociation concentration $K_d$ are very important and have the physical meaning in this system. So it is very meaningful to learn there parameter using machine learning method.}
We learn these parameters from the observation data through optimising the loss function \eqref{Los_Fun} with $\lambda_d=1$ with learning rate $10^{-3}$. The observation data are given in Fig.\ref{tfa_mptp}(a) and we uniformly choose the number of the observation data $N=21$. The parameter evolution predictions as the iteration of the optimizer progresses for transition times $T=5,7,10$ are shown in Fig.\ref{1D_bio_fw_11}(a1)-(a3). Black curves are true parameters with $K_d=10$, $k_f=6$, $k_d=1$ and $R_{bas}=0.6$. Green curves are learned by Markovian bridge process observation data and red curves are computed by neural network observation data. We see that the learned parameters are very close to the real parameters with an error of less than $5\%$ in both types of observation data. Fig.\ref{1D_bio_fw_11}(b1)-(b3) present the learned results of the drift functions. It shows that the PINNs method can effectively recover the stochastic dynamical structures for both neural network observation data and Markovian bridge process observation data. 

\begin{figure}
\begin{minipage}[]{0.3 \textwidth}
 \leftline{~~~~~~~\tiny\textbf{(a1)}}
\centerline{\includegraphics[width=4.2cm,height=4.2cm]{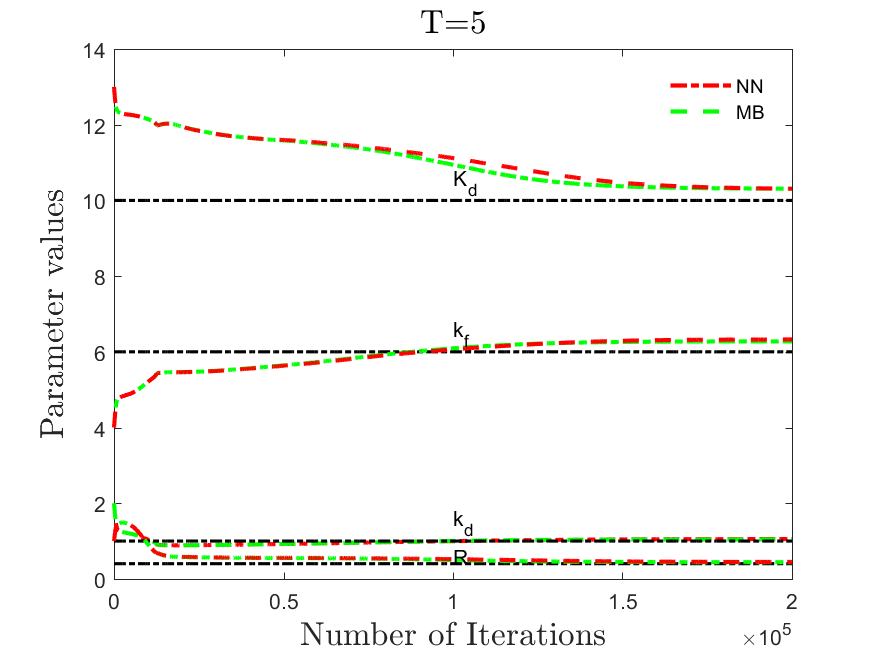}}
\end{minipage}
\hfill
\begin{minipage}[]{0.3 \textwidth}
 \leftline{~~~~~~~\tiny\textbf{(a2)}}
\centerline{\includegraphics[width=4.2cm,height=4.2cm]{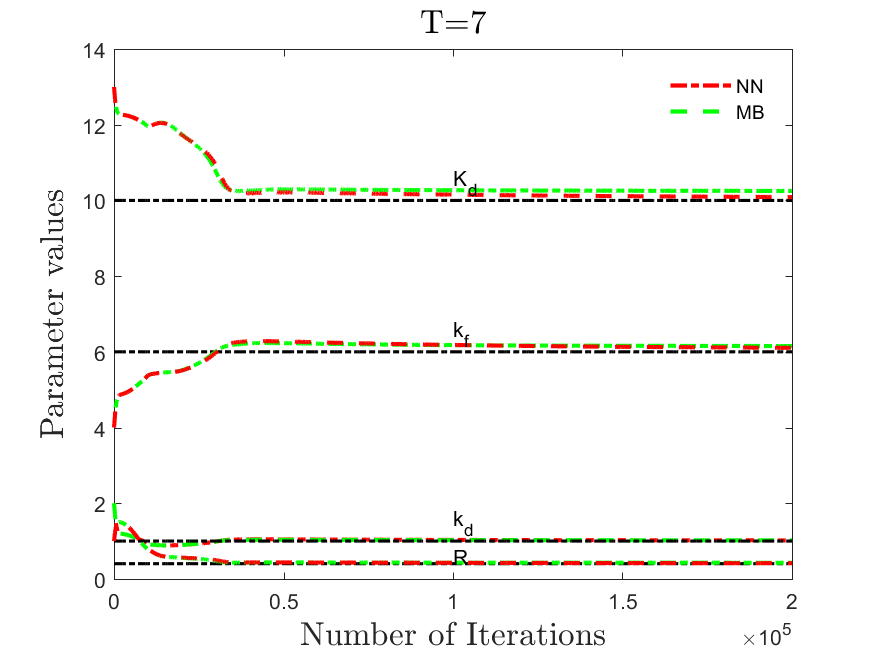}}
\end{minipage}
\hfill
\begin{minipage}[]{0.3 \textwidth}
 \leftline{~~~~~~~\tiny\textbf{(a3)}}
\centerline{\includegraphics[width=4.2cm,height=4.2cm]{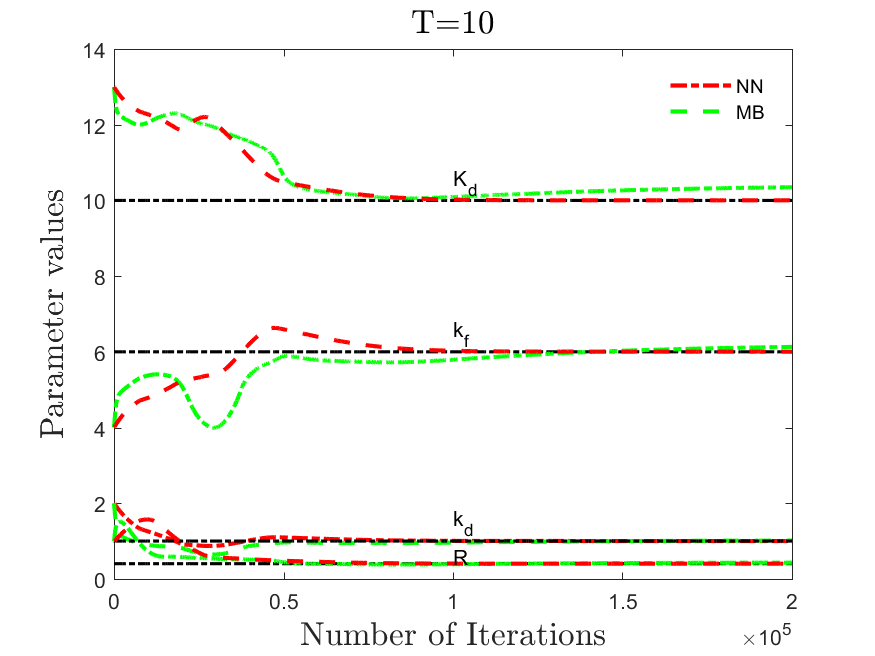}}
\end{minipage}
\begin{minipage}[]{0.3 \textwidth}
 \leftline{~~~~~~~\tiny\textbf{(b1)}}
\centerline{\includegraphics[width=4.2cm,height=4.2cm]{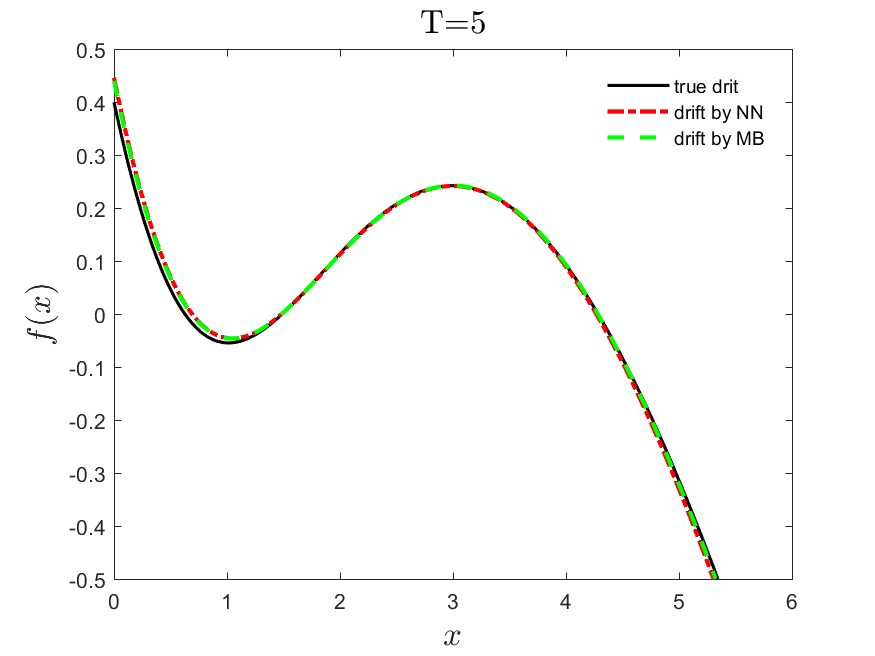}}
\end{minipage}
\hfill
\begin{minipage}[]{0.3 \textwidth}
 \leftline{~~~~~~~\tiny\textbf{(b2)}}
\centerline{\includegraphics[width=4.2cm,height=4.2cm]{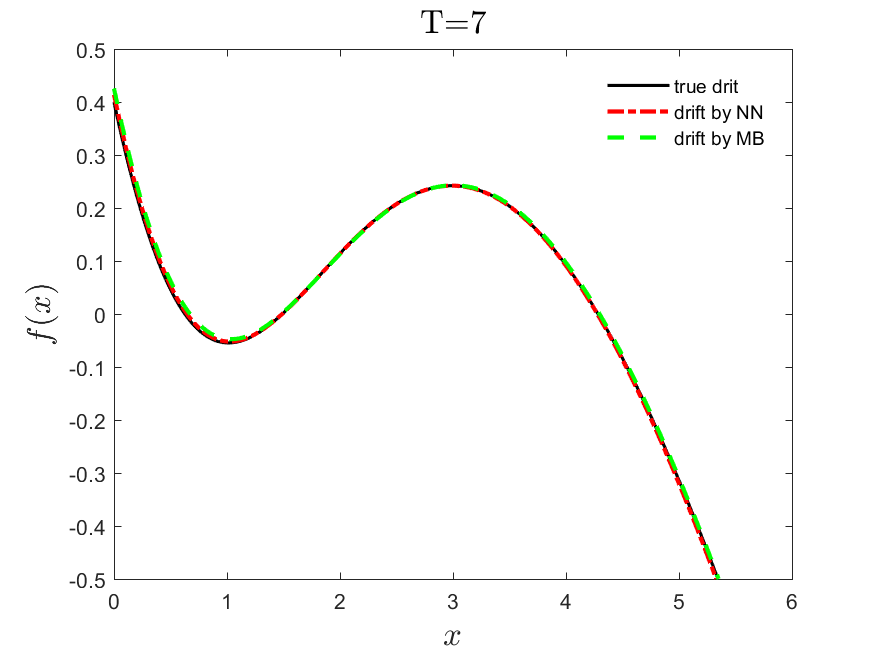}}
\end{minipage}
\hfill
\begin{minipage}[]{0.3 \textwidth}
 \leftline{~~~~~~~\tiny\textbf{(b3)}}
\centerline{\includegraphics[width=4.2cm,height=4.2cm]{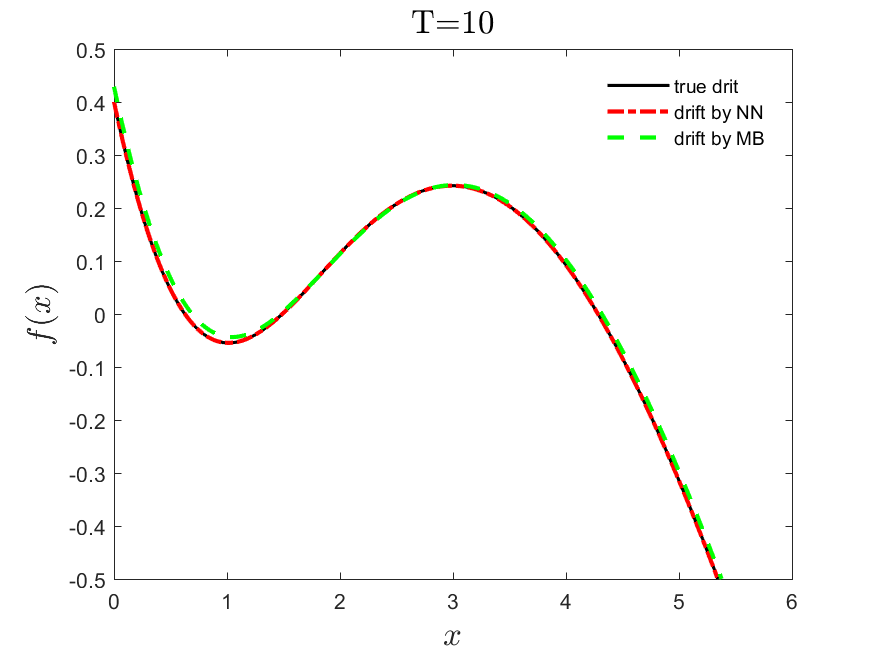}}
\end{minipage}
\caption{\textbf{Parametric estimation in Freidlin-Wentzell case ($\sigma\ll1$) for transition times $T=5,7,10$ - stochastic gene regulation model model:} (a1)-(a3) learned parameters $k_d$, $k_f$, $K_d$ and $R$; (b1)-(b3) learned drift functions. Black curves: true parameters and true drift functions; green curves: learned results for Markovian bridge process observation data; red curves: learned results for neural network observation data.}
\label{1D_bio_fw_11}
\end{figure}

We also do the nonparametric estimation for the drift function from the above two types of observation data for different transition times through optimising the loss function \eqref{loss_unknow} with $\gamma_1=10000$ and $\gamma_2=0$. In this part, we use a fully connected neural network to approximate the drift function. For case I, we have the observation data of the transition path with $N=1001$ and no observation data for the drift function ($N_d=0$). For case II, we  have the observation data of the transition path with $N=1001$ and four observation data of the drift function at $x=0,1.2,3,6$ ($N_d=4$). We consider transition times $T=2,5,7,10$ with two types of observation data; see Fig.\ref{tfa_mptp}(a). 
Fig.\ref{1D_bio_unknow_drift}(a1)-(a4) present the learned drift function from the neural network observation data and Fig.\ref{1D_bio_unknow_drift}(b1)-(b4) show the learned drift function from the Markovian bridge observation data. In Fig.\ref{1D_bio_unknow_drift}, black curves are the true drift functions, red curves (Case I) are learned drift functions with zero observation data of the drift function, and blue curves (Case II) are learned drift functions with four observation data of the drift function. For case I (red curves), the drift function could be learned well in the domain $[0.62685,4.28343]$, except for the transition time $T=2$. When $T=2$, it is difficult to learn the drift, as shown in Fig.\ref{1D_bio_unknow_drift} (a1), because a small transition time $T$ will give less information about the stochastic dynamical system. For case II (blue curves), all the cases can be learned well, even outside of the domain $[0.62685,4.28343]$. 

\begin{figure}
\begin{minipage}[]{0.2 \textwidth}
 \leftline{~~~~~~~\tiny\textbf{(a1)}}
\centerline{\includegraphics[width=3.4cm,height=3.4cm]{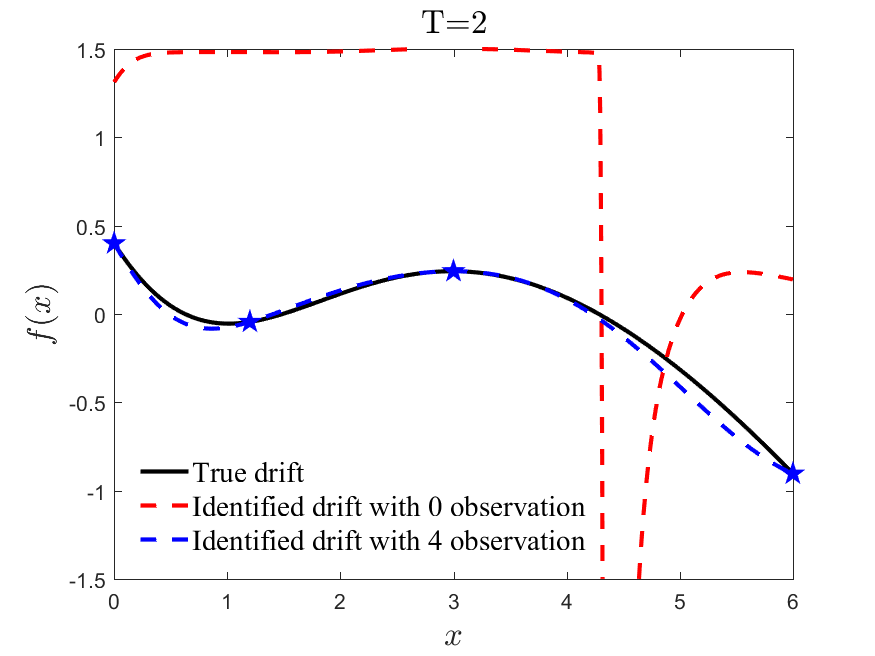}}
\end{minipage}
\hfill
\begin{minipage}[]{0.2 \textwidth}
 \leftline{~~~~~~~\tiny\textbf{(a2)}}
\centerline{\includegraphics[width=3.4cm,height=3.4cm]{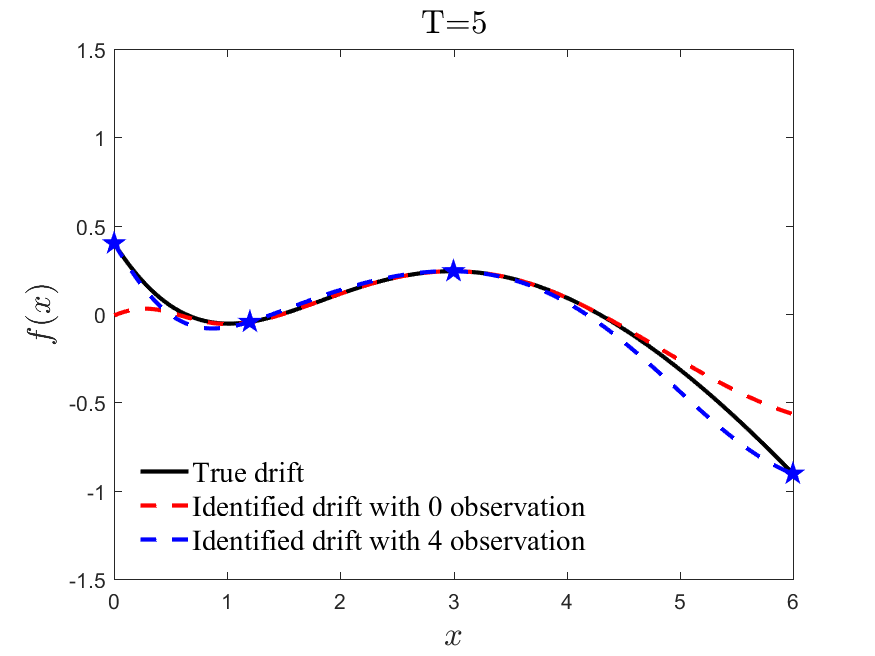}}
\end{minipage}
\hfill
\begin{minipage}[]{0.2 \textwidth}
 \leftline{~~~~~~~\tiny\textbf{(a3)}}
\centerline{\includegraphics[width=3.4cm,height=3.4cm]{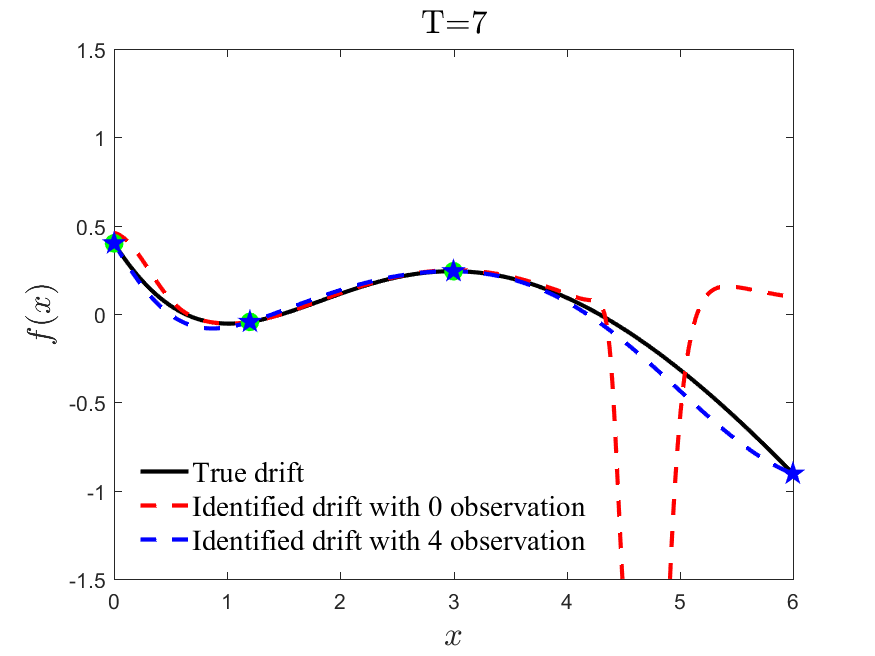}}
\end{minipage}
\hfill
\begin{minipage}[]{0.2 \textwidth}
 \leftline{~~~~~~~\tiny\textbf{(a4)}}
\centerline{\includegraphics[width=3.4cm,height=3.4cm]{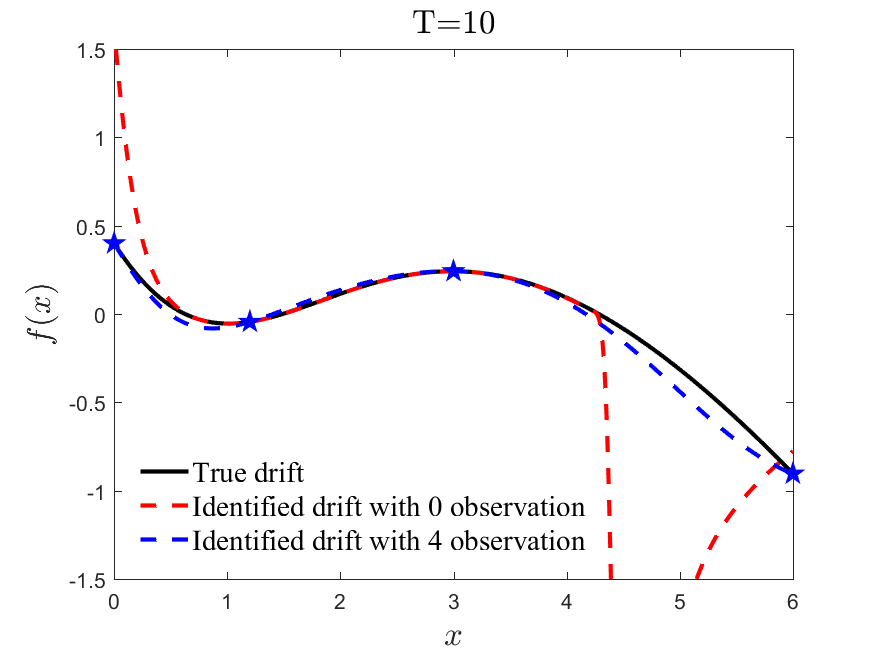}}
\end{minipage}
\hfill
\begin{minipage}[]{0.2 \textwidth}
 \leftline{~~~~~~~\tiny\textbf{(b1)}}
\centerline{\includegraphics[width=3.4cm,height=3.4cm]{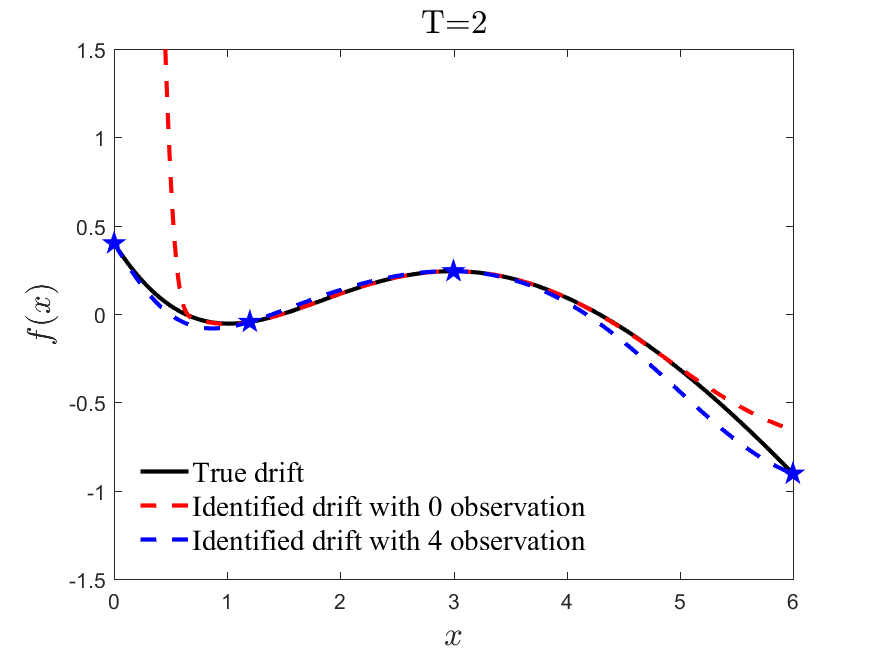}}
\end{minipage}
\hfill
\begin{minipage}[]{0.2 \textwidth}
 \leftline{~~~~~~~\tiny\textbf{(b2)}}
\centerline{\includegraphics[width=3.4cm,height=3.4cm]{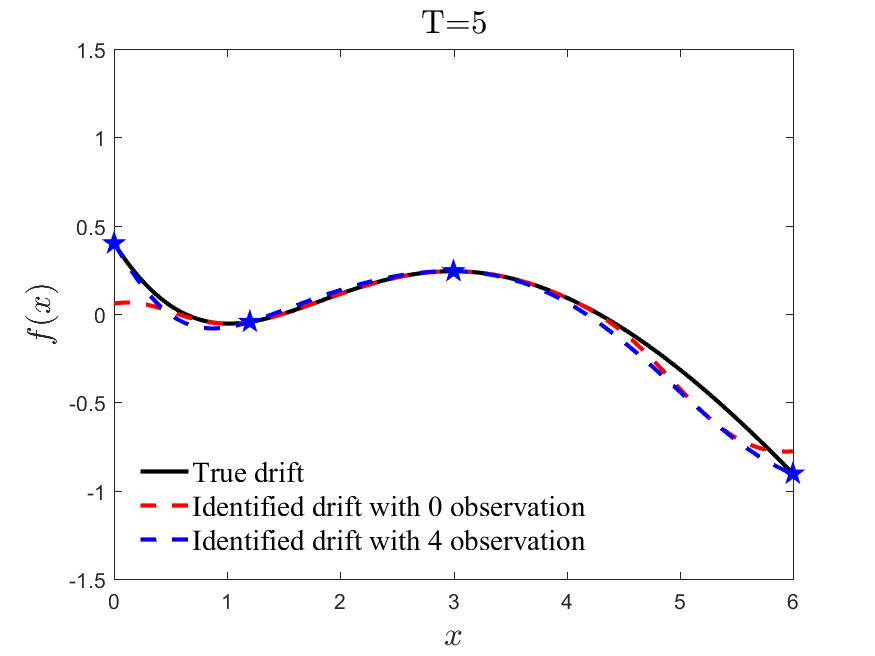}}
\end{minipage}
\hfill
\begin{minipage}[]{0.2 \textwidth}
 \leftline{~~~~~~~\tiny\textbf{(b3)}}
\centerline{\includegraphics[width=3.4cm,height=3.4cm]{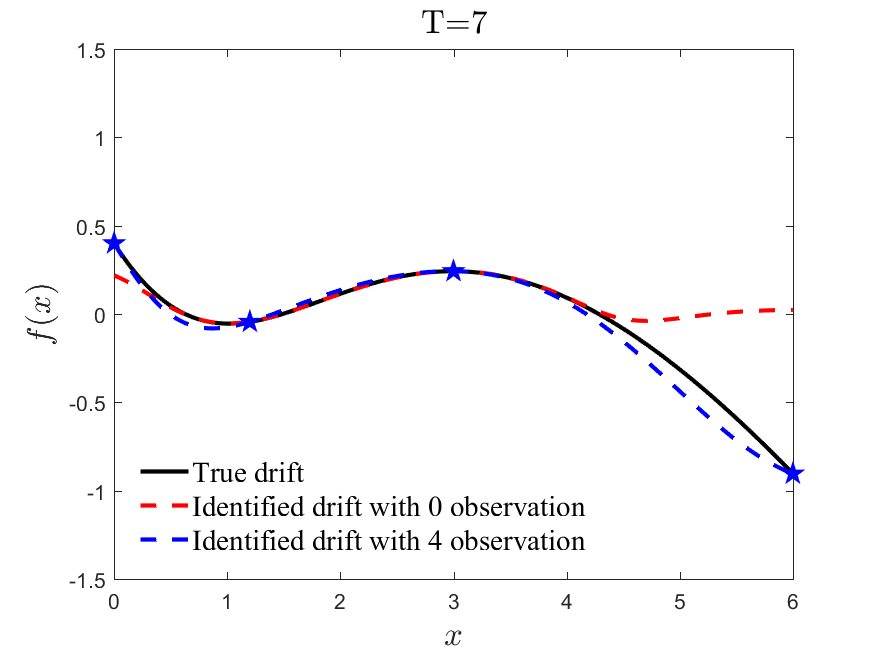}}
\end{minipage}
\hfill
\begin{minipage}[]{0.2 \textwidth}
 \leftline{~~~~~~~\tiny\textbf{(b4)}}
\centerline{\includegraphics[width=3.4cm,height=3.4cm]{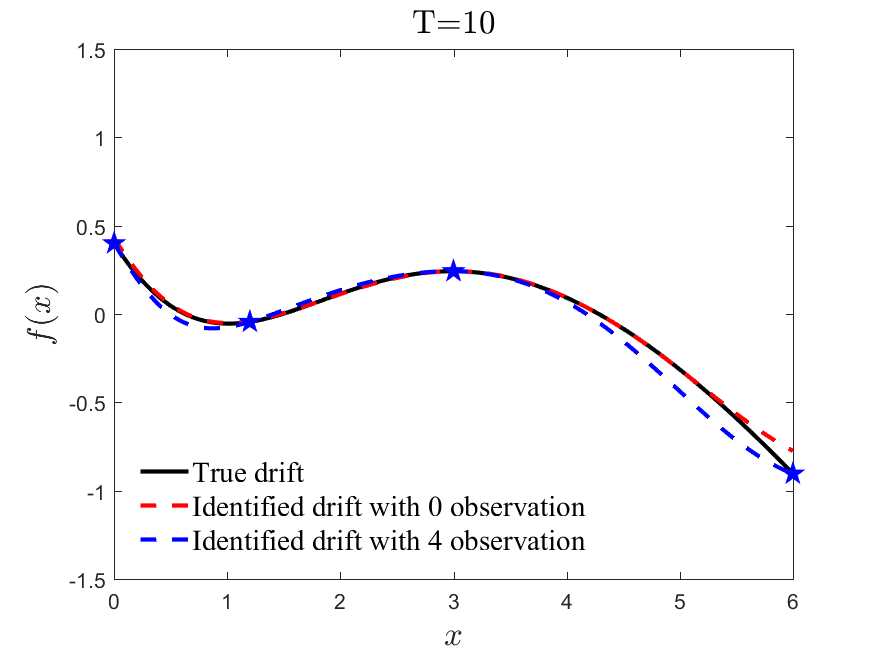}}
\end{minipage}
\caption{\textbf{Non-parametric estimation in Freidlin-Wentzell case ($\sigma\ll 1$) for transition times $T=2,5,7,10$ - stochastic gene regulation model model:} (a) neural network observation data; (b) Markovian bridge observation data. Black curves: true drift function; red curves: learned drift function with zero observation data of drift function; blue curves: learned drift function with four observation data of drift function; blue stars: observation data of drift function.}
\label{1D_bio_unknow_drift}
\end{figure}

\noindent \textbf{Onsager-Machlup framework}

The Onsager-Machlup framework is similar to the Freidlin-Wentzell framework, where the Euler-Lagrange equation is 
\begin{equation}
\begin{split}\label{EL_tfa_om}
\ddot{z}=(\frac{k_fz^2}{z^2+K_d}-k_dz+R_{bas})(\frac{2k_fK_dz}{(z^2+K_d)^2}-k_d)+\sigma^2\frac{k_fK_d(K_d-3z^2)}{(u^2+K_d)^3},
\end{split}
\end{equation}
with two boundary points $z(0)=0.62685$ and $z(T)=4.28343$. Here, we take the noise intensity $\sigma=0.1$.

The most probable transition pathways between $z(0)=0.62685$ and $z(T)=4.28343$ are shown in Fig.\ref{tfa_mptp}(b). Red curves are computed by the neural network to the Euler-Lagrange equation \eqref{EL_tfa_om}, through optimizing the loss function \eqref{emp_loss} with $\lambda_r=1$ and $\lambda_b=1$. Green curves are computed by the Markovian bridge process \eqref{MarBri1}. As shown in Fig.\ref{tfa_mptp}(b), for small transition times, they coincide well.

Suppose the parameters $k_f, K_d, k_d, R_{bas}$ are unknown in \eqref{tfa_pot}. We learn the parameters of the drift function from the observation data through optimising the loss function \eqref{Los_Fun} with $\lambda_d=1$. The observation data are shown in Fig.\ref{tfa_mptp}(b) and we uniformly choose the number of the observation data $N=51$.
Fig.\ref{1D_bio_om_11} shows the results of learned drift functions using the above two types of observation data for transition times $T=5,7,10$. Black curves are true drift functions with true parameters $K_d=10$, $k_f=6$, $k_d=1$ and $R_{bas}=0.6$, red curves are computed by the neural network observation data and green curves are computed by the Markovian bridge process observation data. 
We were able to effectively recover the dynamical structures for the neural network observation data. But for the Markovian bridge process observation data, the learned drift function has larger errors, even though the observation data are similar. It shows that it is sensitive to learning the parameters for this system in the Onsager-Machlup framework.

\begin{figure}
\begin{minipage}[]{0.3 \textwidth}
 \leftline{~~~~~~~\tiny\textbf{(a1)}}
\centerline{\includegraphics[width=4.2cm,height=4.2cm]{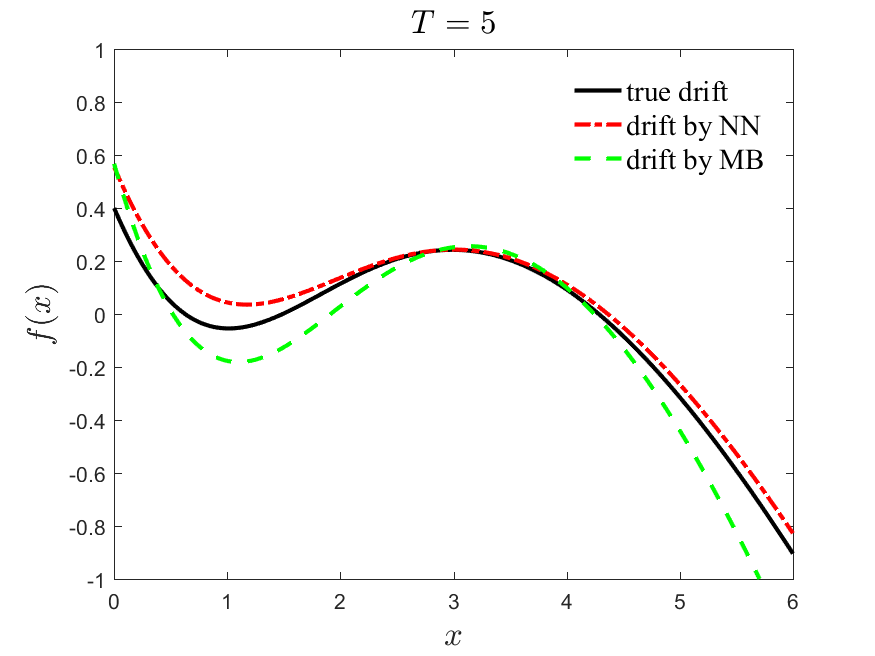}}
\end{minipage}
\hfill
\begin{minipage}[]{0.3 \textwidth}
 \leftline{~~~~~~~\tiny\textbf{(a2)}}
\centerline{\includegraphics[width=4.2cm,height=4.2cm]{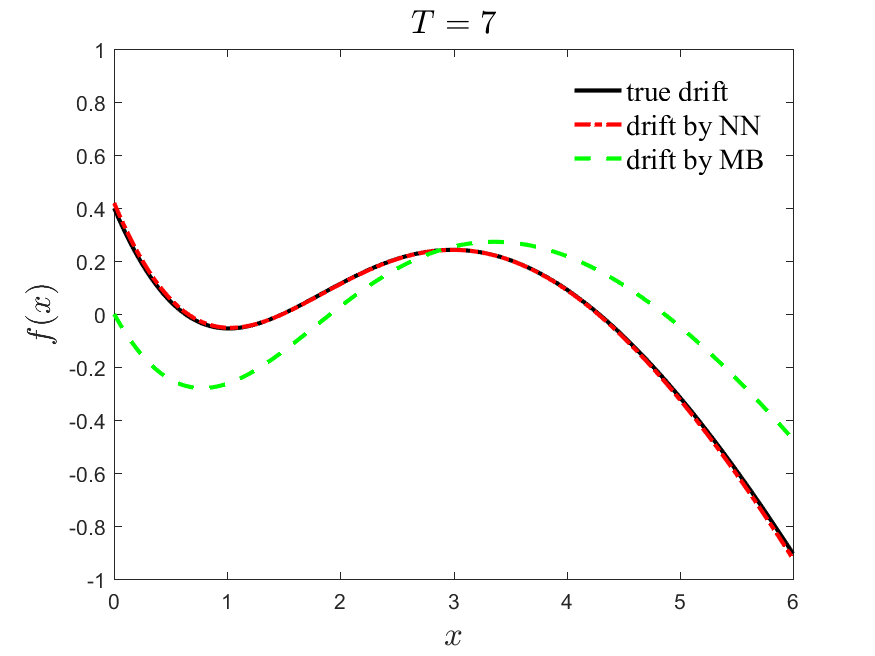}}
\end{minipage}
\hfill
\begin{minipage}[]{0.3 \textwidth}
 \leftline{~~~~~~~\tiny\textbf{(a3)}}
\centerline{\includegraphics[width=4.2cm,height=4.2cm]{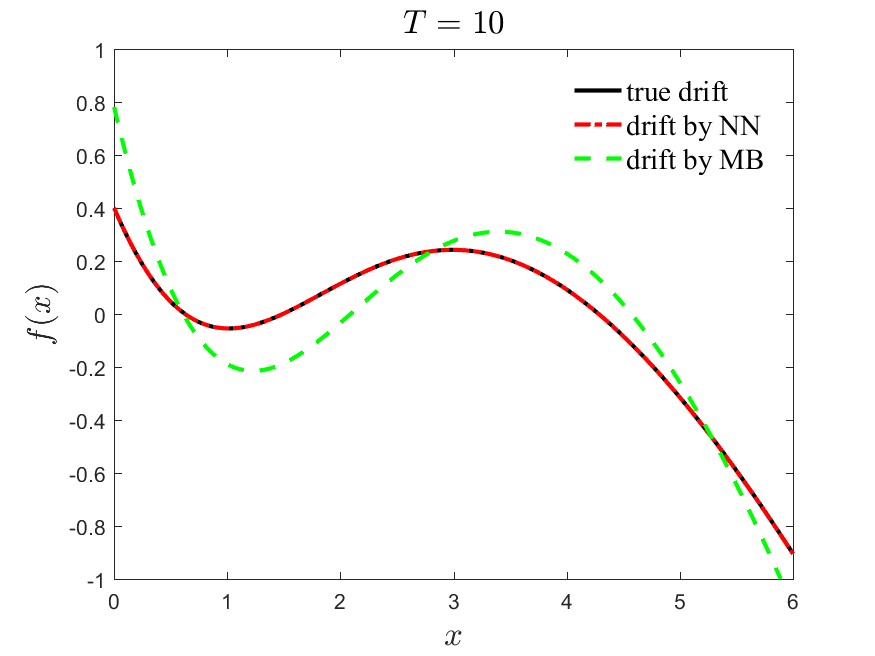}}
\end{minipage}
\caption{\textbf{Parametric estimation in Onsager-Machlup case ($\sigma=0.1$) - stochastic gene regulation model:} learned drift function. (a1) $T=5$; (a2) $T=7$; (a3) $T=10$. Black curves: true drift functions; green curves: learned drift functions for Markovian bridge process observation data; red curves: learned drift functions for neural network observation data.}
\label{1D_bio_om_11}
\end{figure}

We also investigate the nonparametric estimation of the drift function from two types of observation data for different transition times through optimising the loss function \eqref{loss_unknow} with $\gamma_1=10000$. We use a fully connected neural network to approximate the drift function. For case I, we have the observation data of the transition path, but the loss function has no regulization ($N_d$=0 and $\gamma_2=0$). For case II, we have the observation data of the transition path and add the regulization to the loss function ($N_d$=0 and $\gamma_2=10^{-5}$). For case III, we have the observation data of the transition paths and four observation data of the drift function at $x=0,1.2,3,6$ ($N_d$=4 and $\gamma_2=0$). And we consider transition times $T=2,5,7,10$. 
Fig.\ref{1D_unknow_OM_BI}(a1)-(a4) present the learned drift function from the neural network observation data and Fig.\ref{1D_unknow_OM_BI}(b1)-(b4) show the learned drift function from the Markovian bridge observation data. Black curves are true drift functions, red curves (Case I) are learned drift functions with zero observation data of drift function and no regulization in the loss function, green curves (Case II) are learned drift functions with zero observation data of drift function with regulization, and blue curves (Case III) are learned drift functions with four observation data of drift function.
For case I (red curves), the drift functions are learned well for neural network observation data in the domain $[0.62685,4.28343]$, while for Markovian bridge observation data, the drift functions could not be learned well. Then we add the regulization in the loss function to learn the drift. The drift function could be learned well for both observation data in the domain $[0.62685,4.28343]$, except for Markovian bridge observation data with transition time $T=2$ as shown in Fig.\ref{1D_unknow_OM_BI}(b1). 
For case III (blue curves), all the cases can be learned well, even outside of the domain $[0.62685,4.28343]$. The results show the effectiveness of the regulization.

\begin{figure}
\begin{minipage}[]{0.2 \textwidth}
 \leftline{~~~~~~~\tiny\textbf{(a1)}}
\centerline{\includegraphics[width=3.4cm,height=3.4cm]{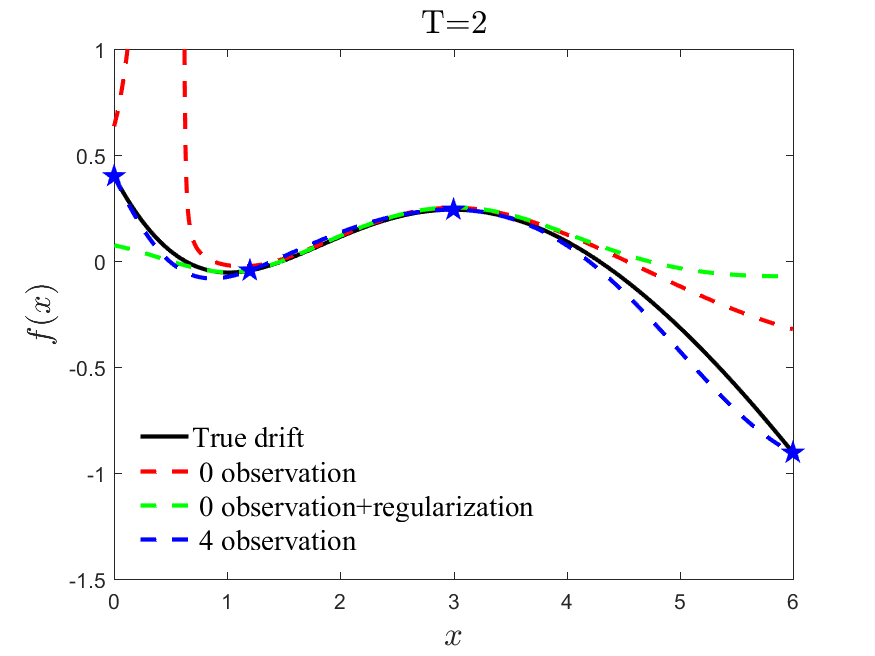}}
\end{minipage}
\hfill
\begin{minipage}[]{0.2 \textwidth}
 \leftline{~~~~~~~\tiny\textbf{(a2)}}
\centerline{\includegraphics[width=3.4cm,height=3.4cm]{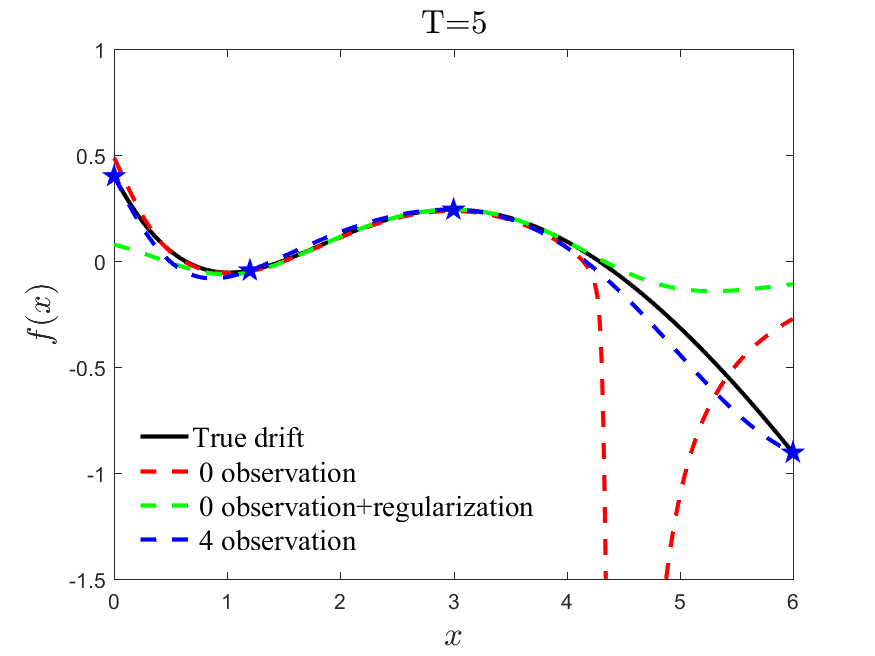}}
\end{minipage}
\hfill
\begin{minipage}[]{0.2 \textwidth}
 \leftline{~~~~~~~\tiny\textbf{(a3)}}
\centerline{\includegraphics[width=3.4cm,height=3.4cm]{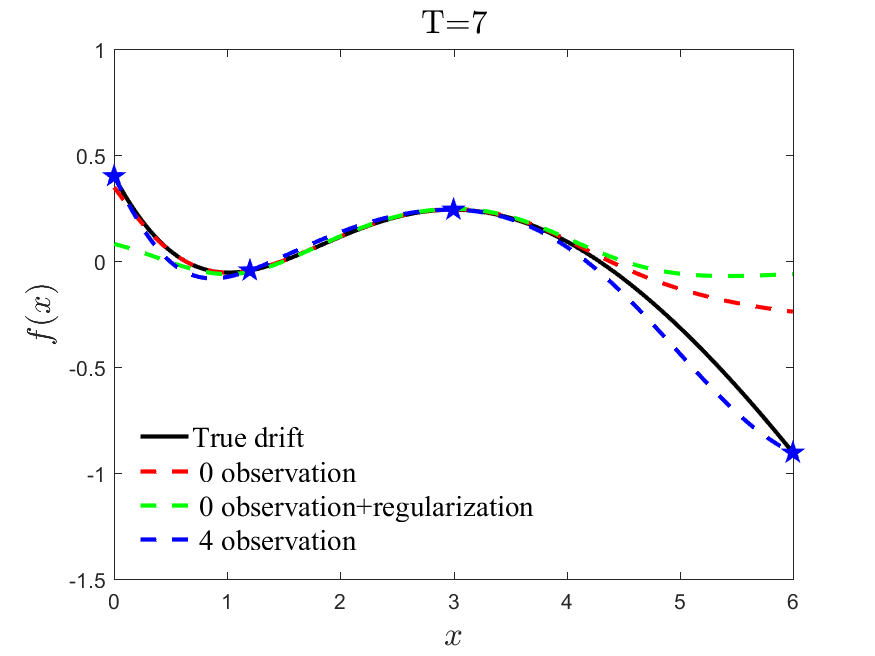}}
\end{minipage}
\hfill
\begin{minipage}[]{0.2 \textwidth}
 \leftline{~~~~~~~\tiny\textbf{(a4)}}
\centerline{\includegraphics[width=3.4cm,height=3.4cm]{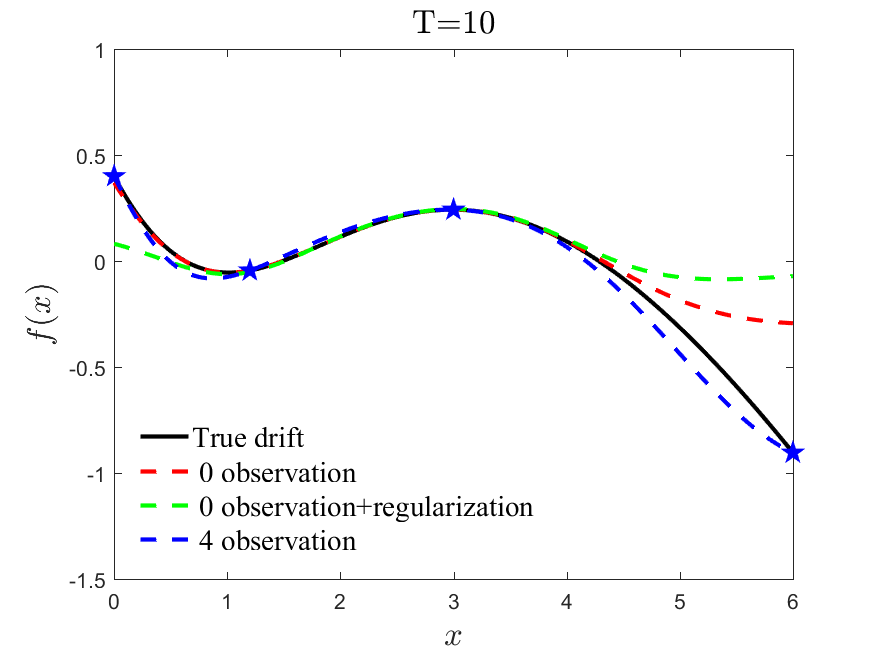}}
\end{minipage}
\hfill
\begin{minipage}[]{0.2 \textwidth}
 \leftline{~~~~~~~\tiny\textbf{(b1)}}
\centerline{\includegraphics[width=3.4cm,height=3.4cm]{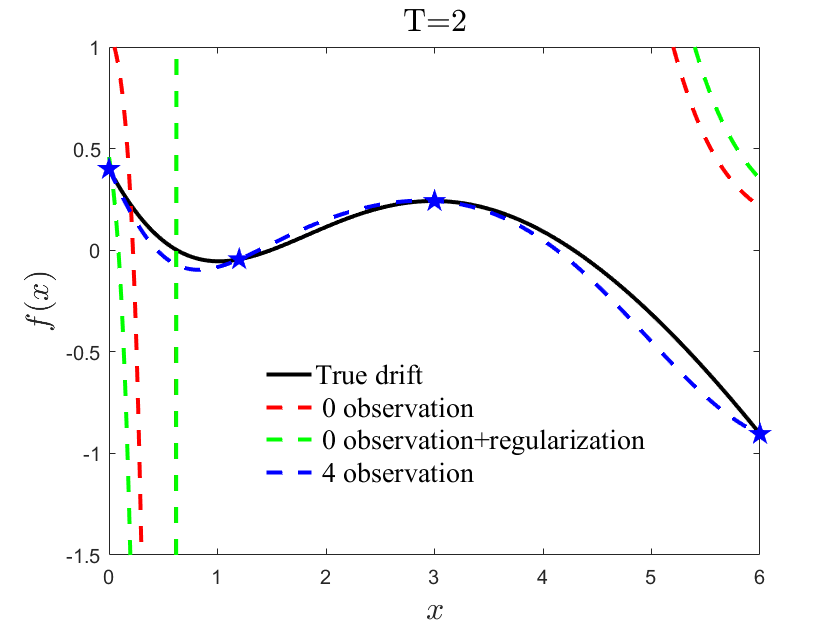}}
\end{minipage}
\hfill
\begin{minipage}[]{0.2 \textwidth}
 \leftline{~~~~~~~\tiny\textbf{(b2)}}
\centerline{\includegraphics[width=3.4cm,height=3.4cm]{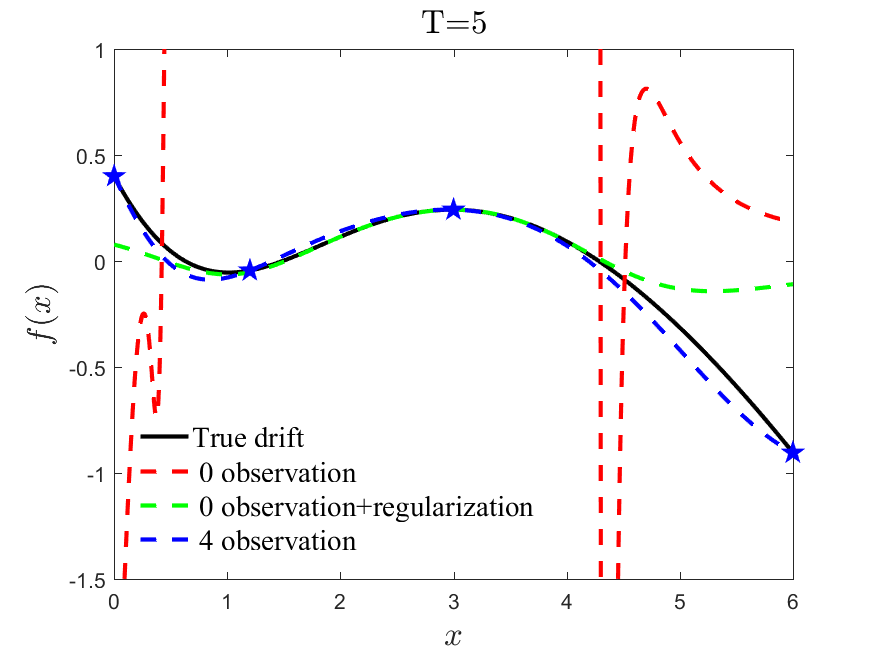}}
\end{minipage}
\hfill
\begin{minipage}[]{0.2 \textwidth}
 \leftline{~~~~~~~\tiny\textbf{(b3)}}
\centerline{\includegraphics[width=3.4cm,height=3.4cm]{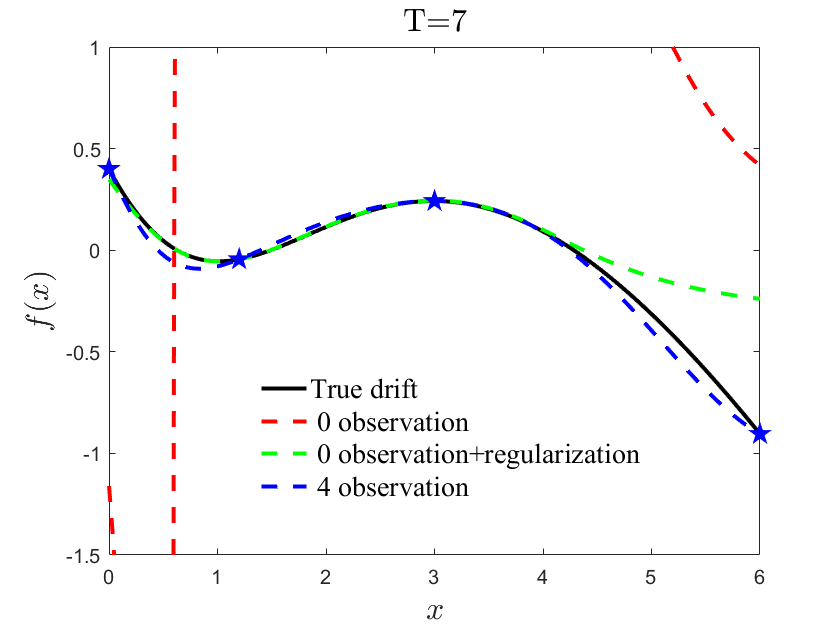}}
\end{minipage}
\hfill
\begin{minipage}[]{0.2 \textwidth}
 \leftline{~~~~~~~\tiny\textbf{(b4)}}
\centerline{\includegraphics[width=3.4cm,height=3.4cm]{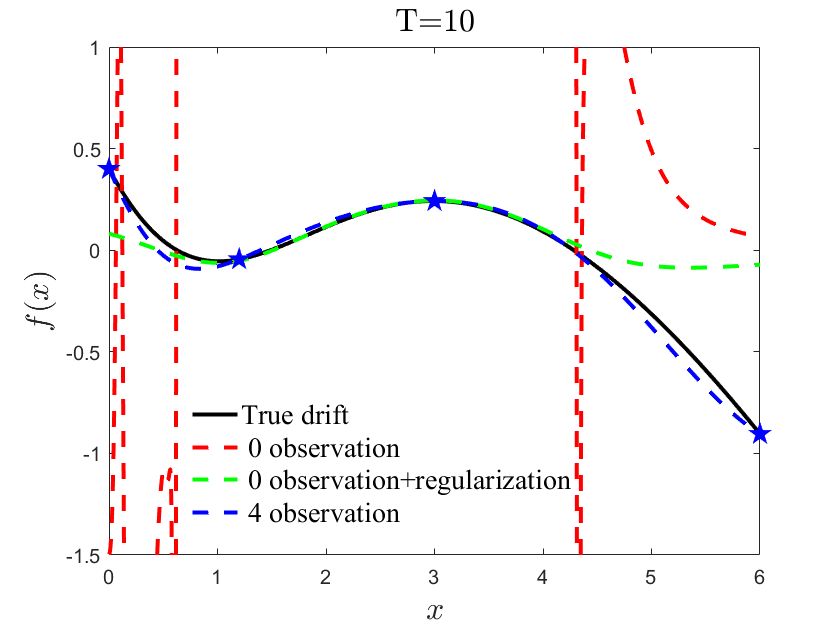} }
\end{minipage}
\caption{\textbf{Non-parametric estimation in Onsager-Machlup case ($\sigma=0.1$) for transition times $T=2,5,7,10$ - stochastic gene regulation model:} (a) neural network observation data; (b) Markovian bridge observation data. Black curves: true drift; red curves: learned drift function with zero observation data of the drift function and no regulization; green curves: learned drift function with zero observation data of the drift function and regulization; blue curves: learned drift function with four observation data of the drift function; blue stars: observation data of the drift function.}
\label{1D_unknow_OM_BI}
\end{figure}


\subsection{Stochastic Maier-Stein system}
Consider the following stochastic Maier\!-Stein system \cite{maier1993escape,maier1996scaling,wei2022optimal}
\begin{equation}
\begin{split}
dx&=(x-x^3-\gamma xy^2)dt+\sigma dW^1(t),\\
dy&=-(1+x^2)ydt+\sigma dW^2(t),
\end{split}
\end{equation}
where $\gamma$ is a positive parameter, $W^1(t)$ and $W^2(t)$ are two independent Brownian motions, and $\sigma$ is the noise intensity.

Note that there exist two stable nodes $(-1,0)$ and $(1,0)$ and one unstable node (0,0) for the corresponding deterministic system. Denote the drift function $f(x,y)=(f^1(x,y),f^2(x,y))^T=(x-x^3-\gamma xy^2,-(1+x^2)y)^T$. It is known that the gradient matrix $Df$ is symmetric if and only if $\gamma=1$. In this case, the potential function of the Maier-Stein system is
\begin{equation}
\begin{split}
V(x,y)=-\frac{\lambda_1}{2}x^2-\frac{\lambda_2}{4}x^4-\frac{\lambda_3}{2}x^2y^2-\frac{\lambda_4}{2}y^2,
\end{split}
\end{equation}
where $\lambda_1=1$, $\lambda_2=\lambda_3=\lambda_4=-1$. \textcolor{red}{For various parameters, the dynamical behavior of the Maier-Stein system will differ depending on $\lambda_i$. If the parameter changes, a bifurcation occurs in this system \cite{maier1996scaling}. So in the following computing, we will infer the parameters $\lambda_i$,} $i=1,2,3,4$ with transition time $T=10$. \textcolor{red}{The neural network has 2 hidden layers and 20 neurons in each layer. }

The Euler-Lagrange equation for $z=(x, y)$ corresponding to the Onsager-Machlup action functional reads
\begin{equation}
\begin{split}\label{2d_EL_om}
\ddot{x}&=\dot{y}\left(\partial_y f^1- \partial_x f^2\right)+f^1 \partial_{x} f^1+ f^2 \partial_{x} f^2+ \frac{\sigma^2}{2}\partial_{x} \partial_x f^1+ \frac{\sigma^2}{2}\partial_x \partial_y f^2, \\
\ddot{y}&=\dot{x}\left(\partial_{x} f^2- \partial y f^1\right)+ f^1 \partial_y f^1+f^2 \partial_{y} f^2+ \frac{\sigma^2}{2}\partial_y \partial_x f^1+ \frac{\sigma^2}{2}\partial_y \partial_y f^2 .
\end{split}
\end{equation}

The Euler-Lagrange equation for $z=(x, y)$ corresponding to the Freidlin-Wentzell action functional reads
\begin{equation}
\begin{split}\label{2d_EL_fw}
\ddot{x}&=\dot{y}\left(\partial_y f^1- \partial_x f^2\right)+f^1 \partial_{x} f^1+ f^2 \partial_{x} f^2, \\
\ddot{y}&=\dot{x}\left(\partial_{x} f^2- \partial y f^1\right)+f^1 \partial_y f^1+f^2 \partial_{y} f^2.
\end{split}
\end{equation}

\noindent \textbf{Freidlin-Wentzell framework}

In this part, we learn the parameters in the drift function using the observation data of the most probable transition pathway. The observation data are given at $\{u(t_i)\}_{i=0}^{51}$, where $t_i=0.2i$ and $i=0,1,2 \cdots 50$.

In the first case, we infer the parameters with a transition path between two metastable states of the system, i.e., from $(-1,0)$ to $(1,0)$. On the left of Fig.\ref{2D_che_FW_stable}(a), the black curve is the true most probable transition pathway computed by the forward problem of the Euler-Lagrange equation. And for the inverse problem, the projection drawing is the learned potential function with 51 observation data points (green stars), and the red curve is the learned most probable transition pathway. We can learn the probable transition pathway well. And we find that the most probable transition pathway in $y$ direction is almost $0$. The error of the potential function is shown in the middle of Fig.\ref{2D_che_FW_stable}(a). The learned parameters are shown in the right of the Fig.\ref{2D_che_FW_stable}(a). The parameters $\lambda_1$ and $\lambda_2$ can be learned very well, while the parameters $\lambda_3$ and $\lambda_4$ cannot be learned well. This is because the observation data have no information in $y$ direction and $\lambda_3$ and $\lambda_4$ are the coefficients of $y$. To show this judgement, we then only learn two parameters $\lambda_1$ and $\lambda_2$. The results are shown in Fig.\ref{2D_che_FW_stable}(b). We can see the error of the potential is very small, as shown in the middle of Fig.\ref{2D_che_FW_stable}(b). The parameters $\lambda_1$ and $\lambda_2$ would converge to the true parameters after $50,000$ iteration step.

\begin{figure}
\begin{minipage}[]{0.3 \textwidth}
 \leftline{~~~~~~~\tiny\textbf{(a)}}
\centerline{\includegraphics[width=4.2cm,height=4.2cm]{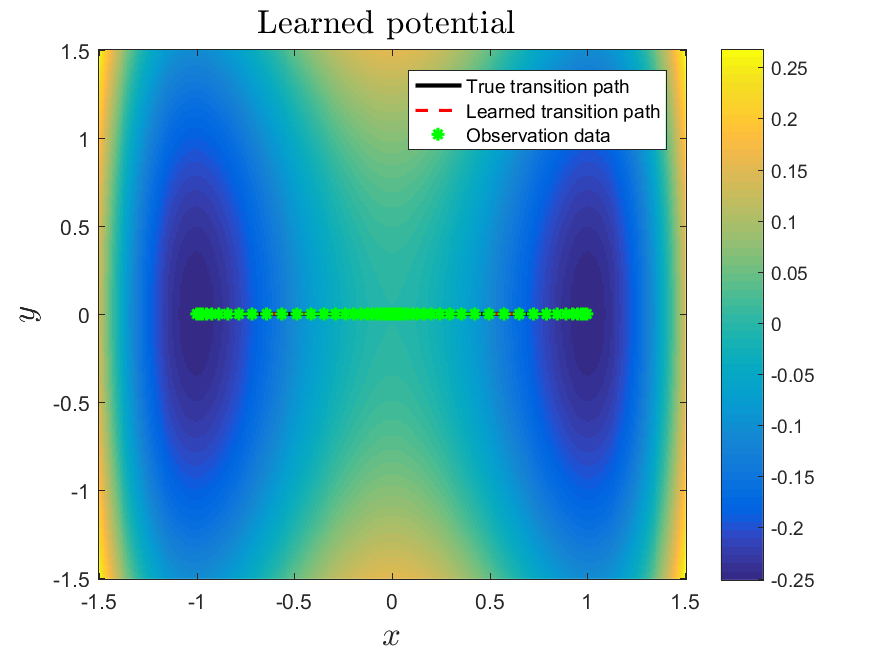}}
\end{minipage}
\hfill
\begin{minipage}[]{0.3 \textwidth}
\centerline{\includegraphics[width=4.2cm,height=4.2cm]{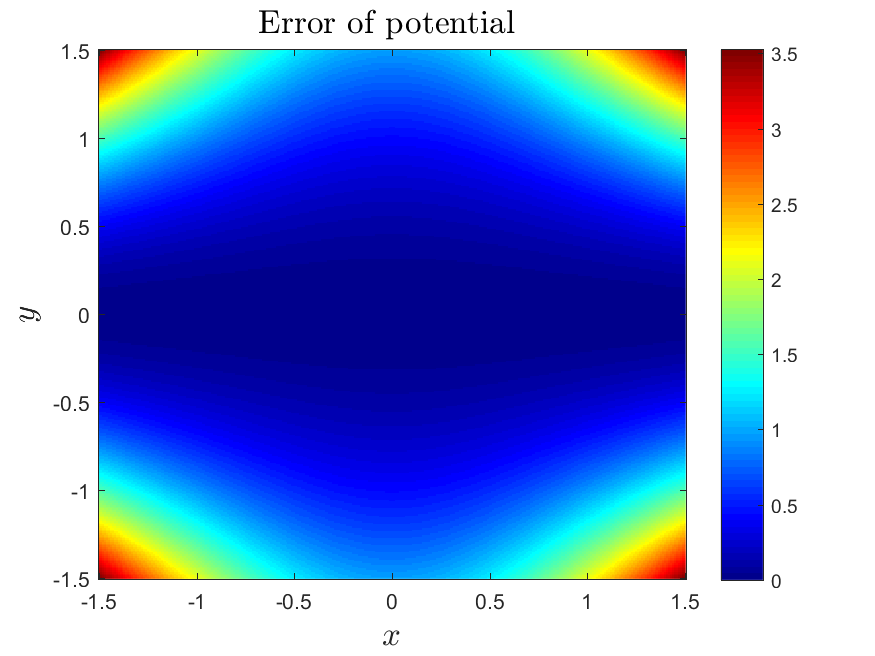}}
\end{minipage}
\hfill
\begin{minipage}[]{0.3 \textwidth}
\centerline{\includegraphics[width=4.2cm,height=4.2cm]{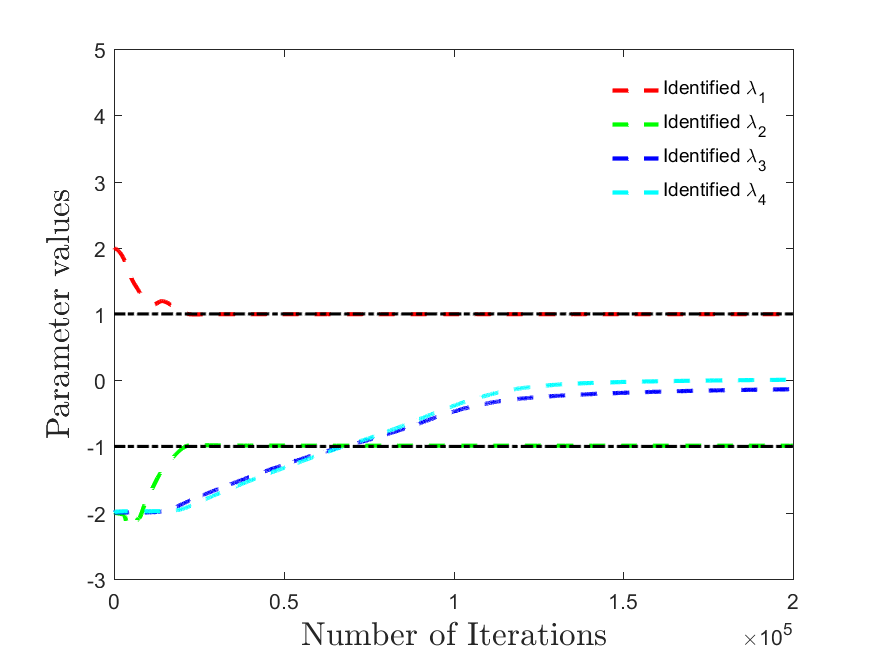}}
\end{minipage}
\begin{minipage}[]{0.3 \textwidth}
 \leftline{~~~~~~~\tiny\textbf{(b)}}
\centerline{\includegraphics[width=4.2cm,height=4.2cm]{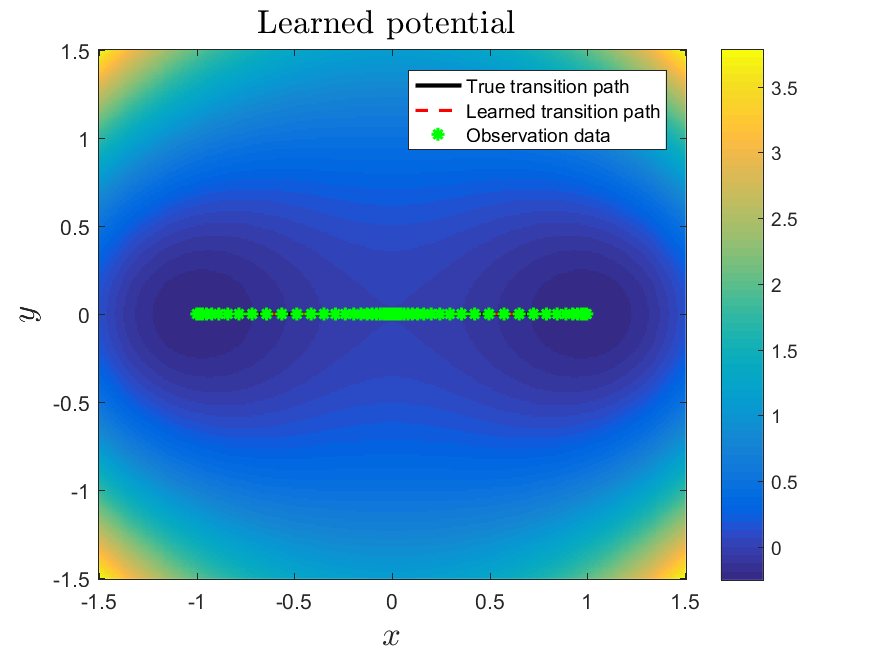}}
\end{minipage}
\hfill
\begin{minipage}[]{0.3 \textwidth}
\centerline{\includegraphics[width=4.2cm,height=4.2cm]{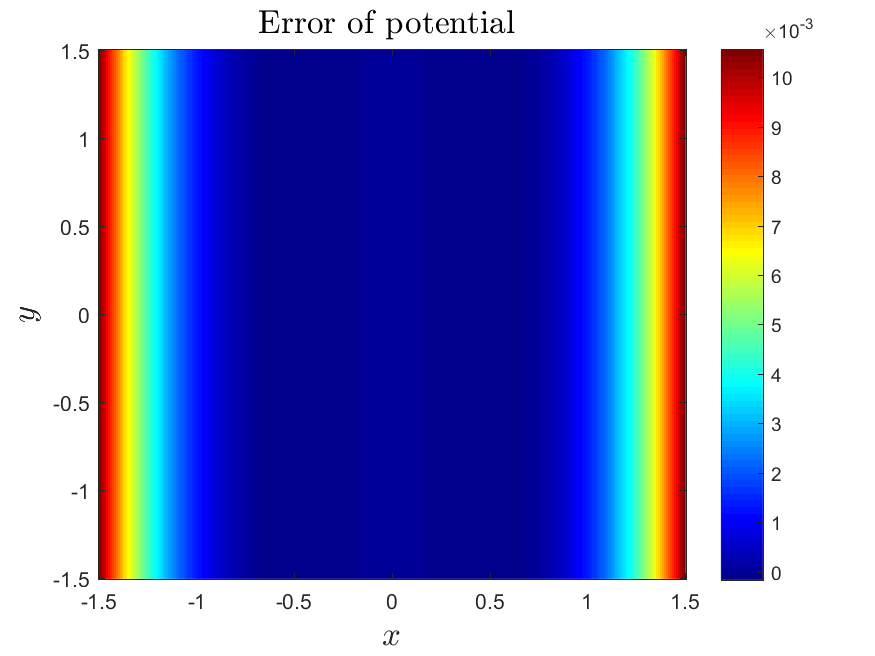}}
\end{minipage}
\hfill
\begin{minipage}[]{0.3 \textwidth}
\centerline{\includegraphics[width=4.2cm,height=4.2cm]{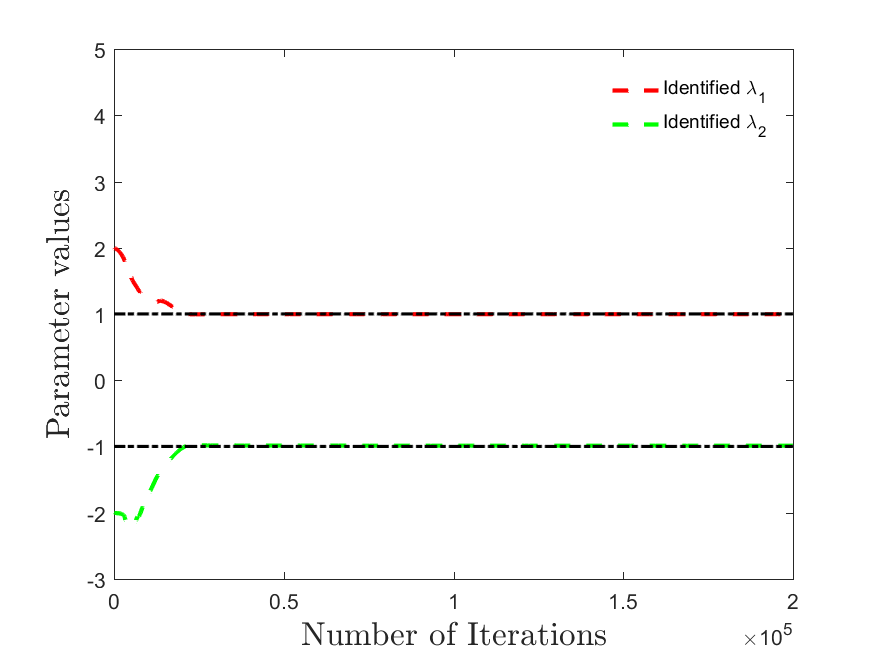}}
\end{minipage}
\caption{\textbf{Parametric estimation for transition time $T=10$ in Freidlin-Wentzell case ($\sigma \ll 1$) - stochastic Maier-Stein model:} (a) learned parameters $\lambda_1$, $\lambda_2$, $\lambda_3$ and $\lambda_4$; (b) learned parameters $\lambda_1$ and $\lambda_2$.  Left: learned most probable transition pathway and potential function; Middle: error of potential; Right: learned parameters.}
\label{2D_che_FW_stable}
\end{figure}

Then we want to explore whether we could learn the system well from the observation data, which has some information in $y$ direction. We consider that the transition path is from $(-1,-1)$ to $(0,0)$, as shown in the black curves in the below of Fig.\ref{2D_che_FW}. We also investigate noisy observation data $u_{ob}=u_{true}(1+\eta N(0,1))$ (the green stars in the below of Fig.\ref{2D_che_FW} ), where $\eta=2\%,5\%, 10\%$ and $N(0,1)$ is normal distribution. We learn the parameters in the potential function and most probable transition pathway with the observation data (green stars in Fig.\ref{2D_che_FW}(b)). Fig.\ref{2D_che_FW}(a1)-(a4) present the learned parameters in each case and Fig.\ref{2D_che_FW}(b1)-(b4) show the learned potential function and most transition pathway (red curves). Comparing with the true most transition pathways (black curves in Fig.\ref{2D_che_FW}(b)), we can learn the most transition pathways well.
The error of potential is small with the clean or true observation data. The parameters can be also learned well even for noisy observation data. While the error would be larger with larger noise observation data, as shown in Fig.\ref{2D_che_FW}(b2)-(b3). If the observation data had information in both x and y directions, we could learn all the parameters well.

\begin{figure}
\begin{minipage}[]{0.2 \textwidth}
 \leftline{~~~~~~~\tiny\textbf{(a1)}}
\centerline{\includegraphics[width=3.4cm,height=3.4cm]{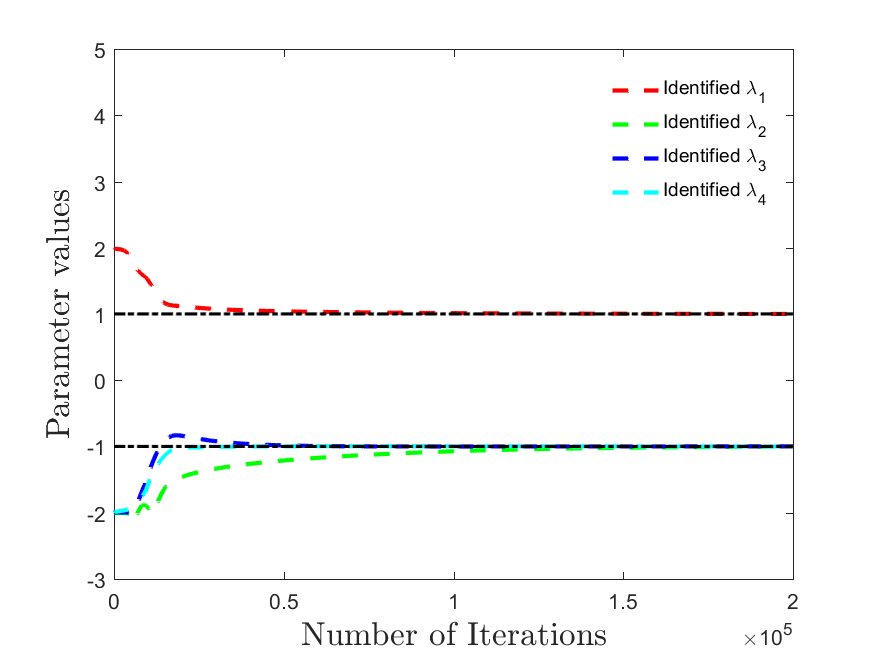}}
\end{minipage}
\hfill
\begin{minipage}[]{0.2 \textwidth}
 \leftline{~~~~~~~\tiny\textbf{(a2)}}
\centerline{\includegraphics[width=3.4cm,height=3.4cm]{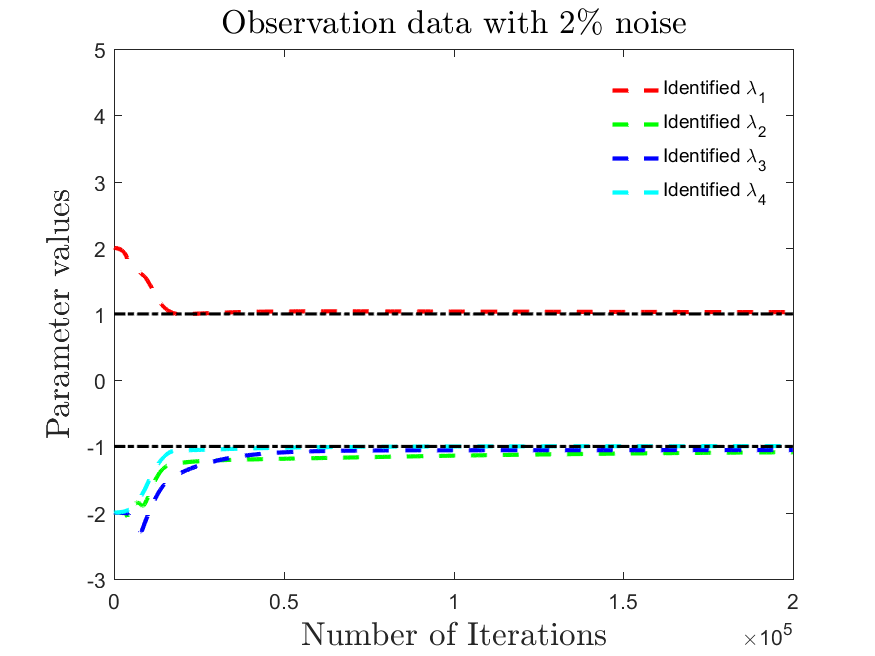}}
\end{minipage}
\hfill
\begin{minipage}[]{0.2 \textwidth}
 \leftline{~~~~~~~\tiny\textbf{(a3)}}
\centerline{\includegraphics[width=3.4cm,height=3.4cm]{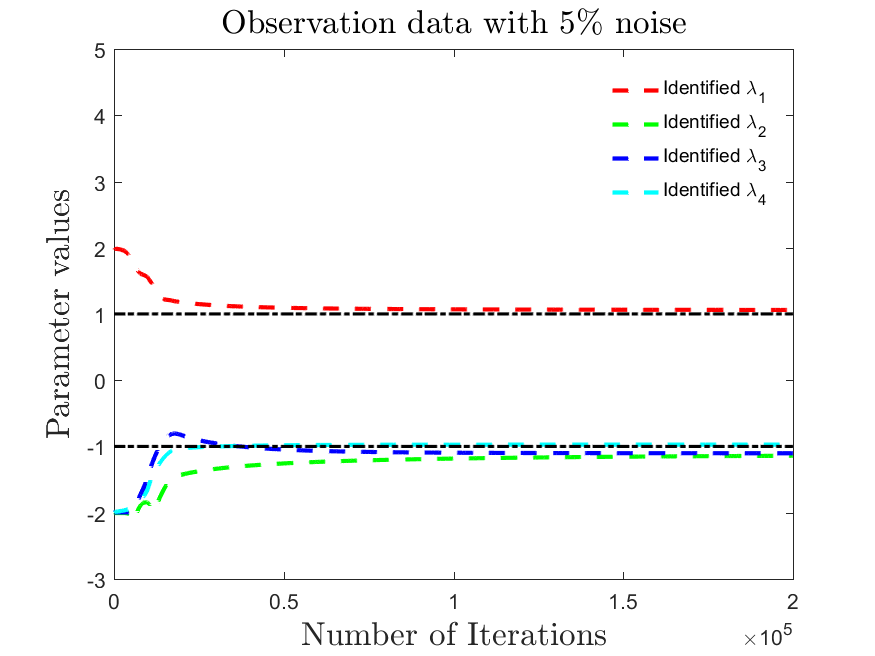}}
\end{minipage}
\hfill
\begin{minipage}[]{0.2 \textwidth}
 \leftline{~~~~~~~\tiny\textbf{(a4)}}
\centerline{\includegraphics[width=3.4cm,height=3.4cm]{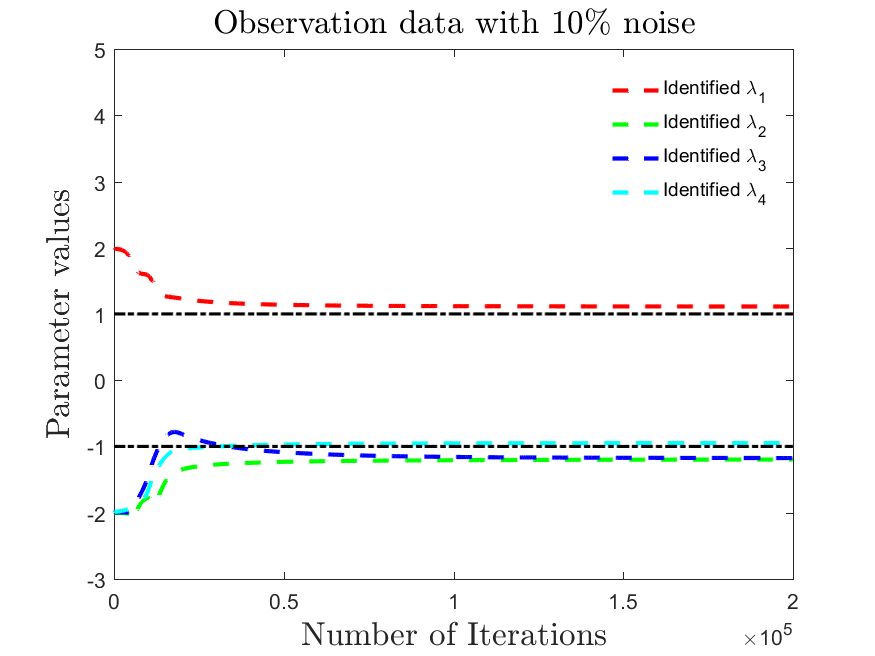}}
\end{minipage}
\hfill
\begin{minipage}[]{0.2 \textwidth}
 \leftline{~~~~~~~\tiny\textbf{(b1)}}
\centerline{\includegraphics[width=3.4cm,height=3.4cm]{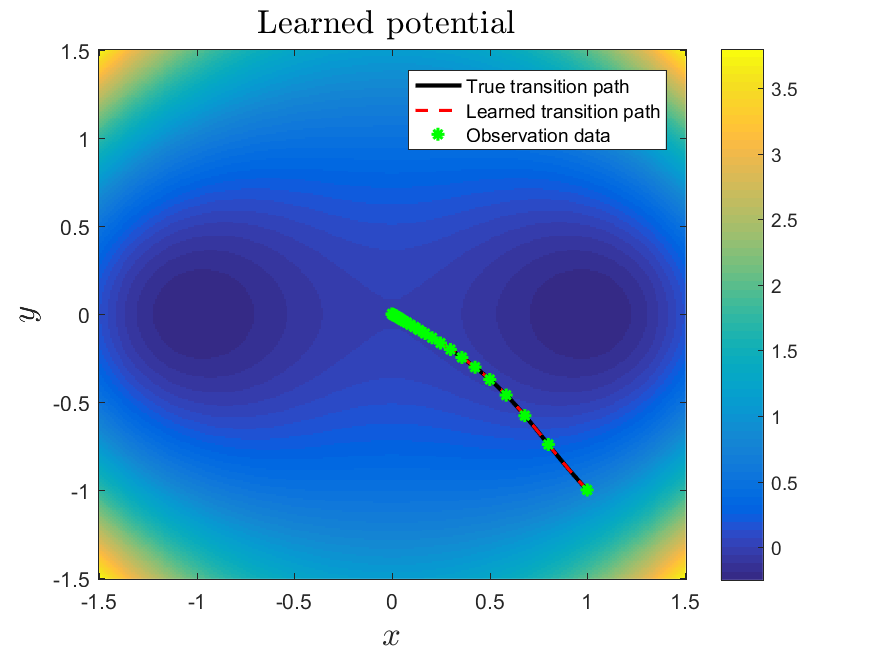}}
\end{minipage}
\hfill
\begin{minipage}[]{0.2 \textwidth}
 \leftline{~~~~~~~\tiny\textbf{(b2)}}
\centerline{\includegraphics[width=3.4cm,height=3.4cm]{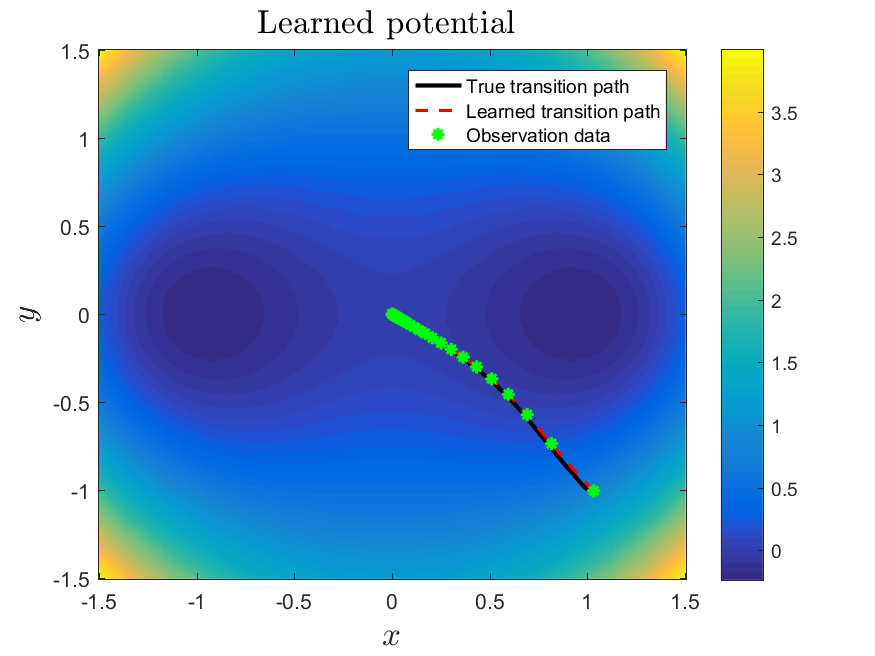}}
\end{minipage}
\hfill
\begin{minipage}[]{0.2 \textwidth}
 \leftline{~~~~~~~\tiny\textbf{(b3)}}
\centerline{\includegraphics[width=3.4cm,height=3.4cm]{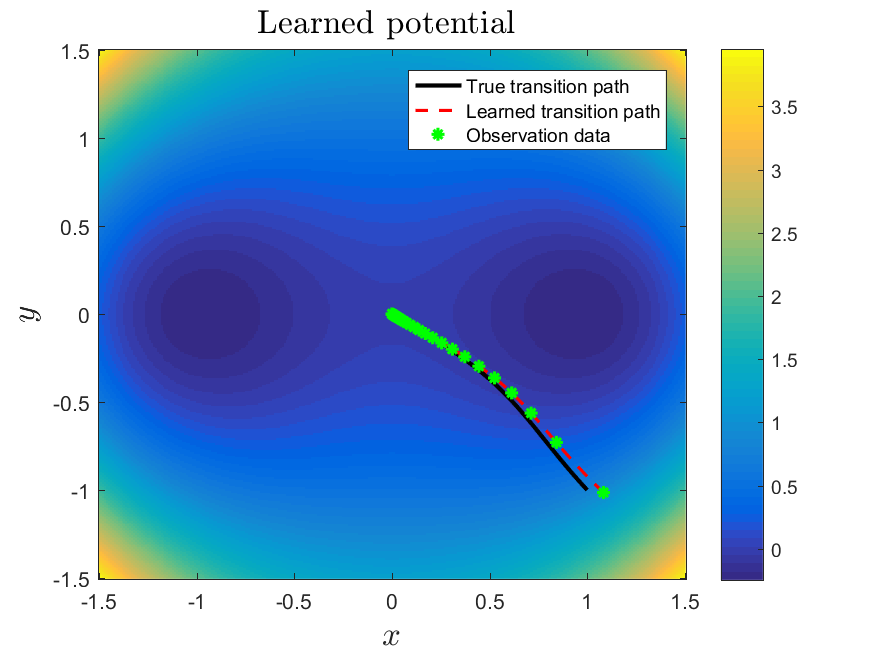}}
\end{minipage}
\hfill
\begin{minipage}[]{0.2 \textwidth}
 \leftline{~~~~~~~\tiny\textbf{(b4)}}
\centerline{\includegraphics[width=3.4cm,height=3.4cm]{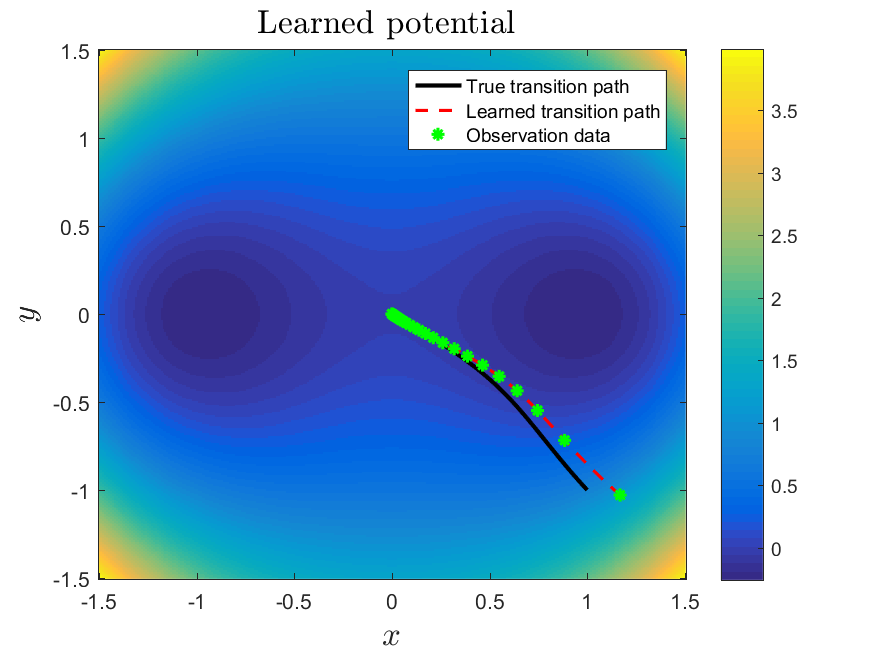}}
\end{minipage}
\caption{\textbf{Parametric estimation for transition time $T=10$ in Freidlin-Wentzell case ($\sigma \ll 1$)  - stochastic Maier-Stein model:} (a1,b1) no noise; (a2,b2) 2\% noise; (a3,b3) 5\% noise; (a4,b4) 10\% noise. Above: learned parameters; below: learned the most probable transition pathway and potential function}
\label{2D_che_FW}
\end{figure}


\noindent \textbf{Onsager-Machlup framework}

In this part, we consider the Onsager-Machlup case with $T=10$ and $\sigma=0.1$. We consider the transition path from $(-1,-1)$ to $(0,0)$ as shown the green stars in Fig.\ref{2D_che_OM}, so that we can learn all the parameters in the drift function. We also consider the noisy observation data. Here the data is $u_{ob}=u_{true}(1+\eta \mathbb{N}(0,1))$, where $\eta=2\%,5\%, 10\%$ and $\mathbb{N}(0,1)$ is normal distribution. The results are shown in Fig.\ref{2D_che_OM}. We can also learn all the parameters in the Onsager-Machlup framework.

\begin{figure}
\begin{minipage}[]{0.2 \textwidth}
 \leftline{~~~~~~~\tiny\textbf{(a1)}}
\centerline{\includegraphics[width=3.4cm,height=3.4cm]{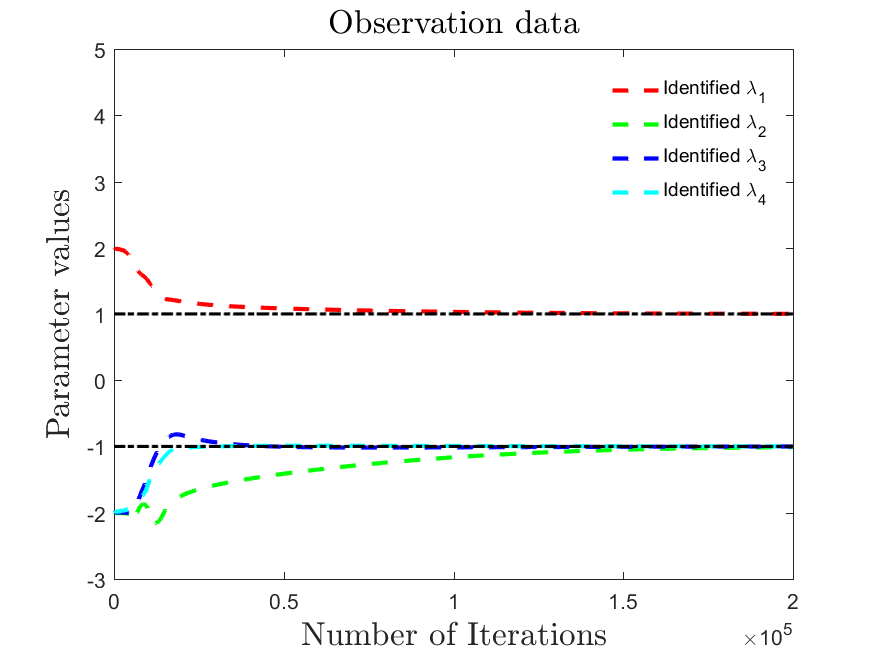}}
\end{minipage}
\hfill
\begin{minipage}[]{0.2 \textwidth}
 \leftline{~~~~~~~\tiny\textbf{(a2)}}
\centerline{\includegraphics[width=3.4cm,height=3.4cm]{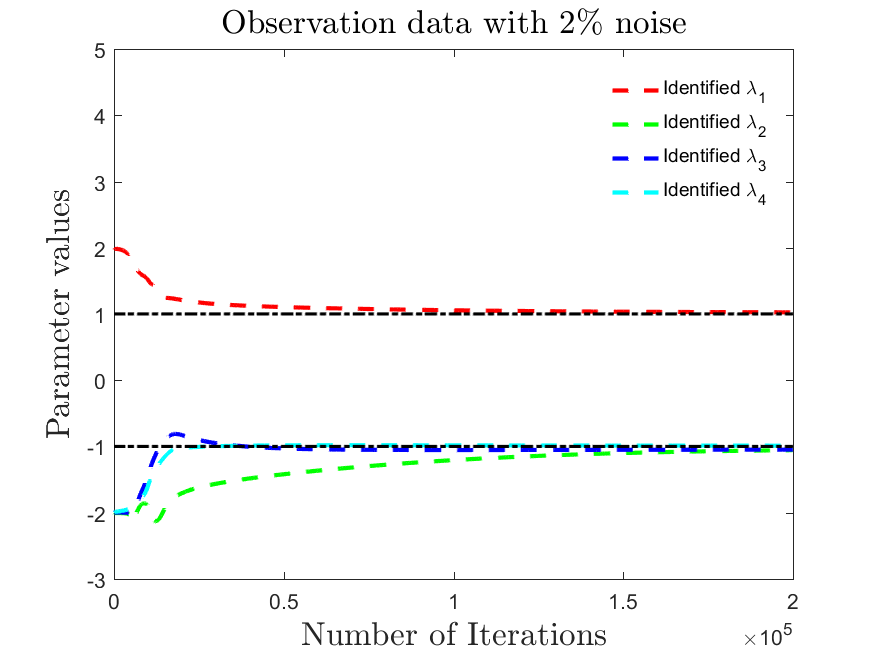}}
\end{minipage}
\hfill
\begin{minipage}[]{0.2 \textwidth}
 \leftline{~~~~~~~\tiny\textbf{(a3)}}
\centerline{\includegraphics[width=3.4cm,height=3.4cm]{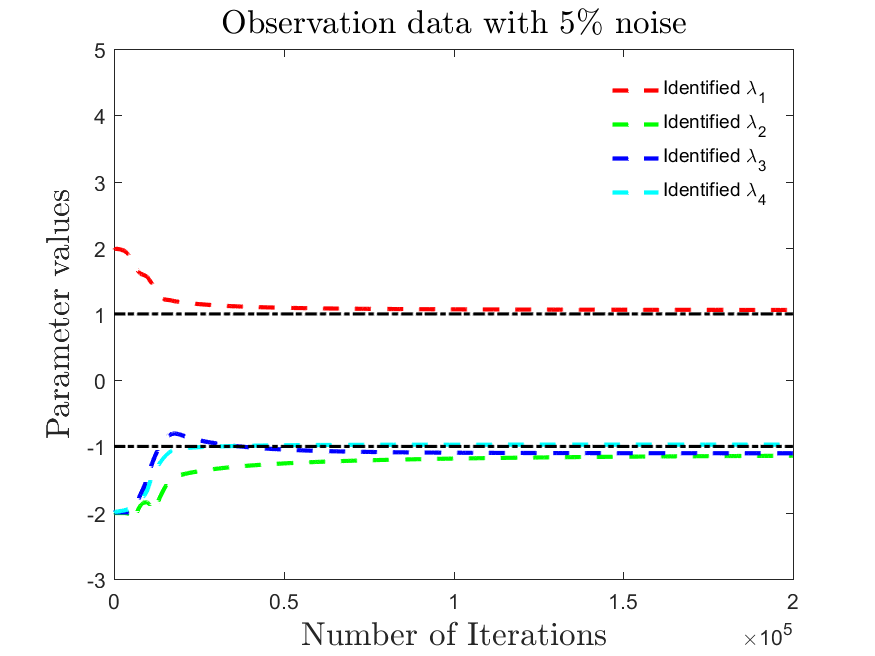}}
\end{minipage}
\hfill
\begin{minipage}[]{0.2 \textwidth}
 \leftline{~~~~~~~\tiny\textbf{(a4)}}
\centerline{\includegraphics[width=3.4cm,height=3.4cm]{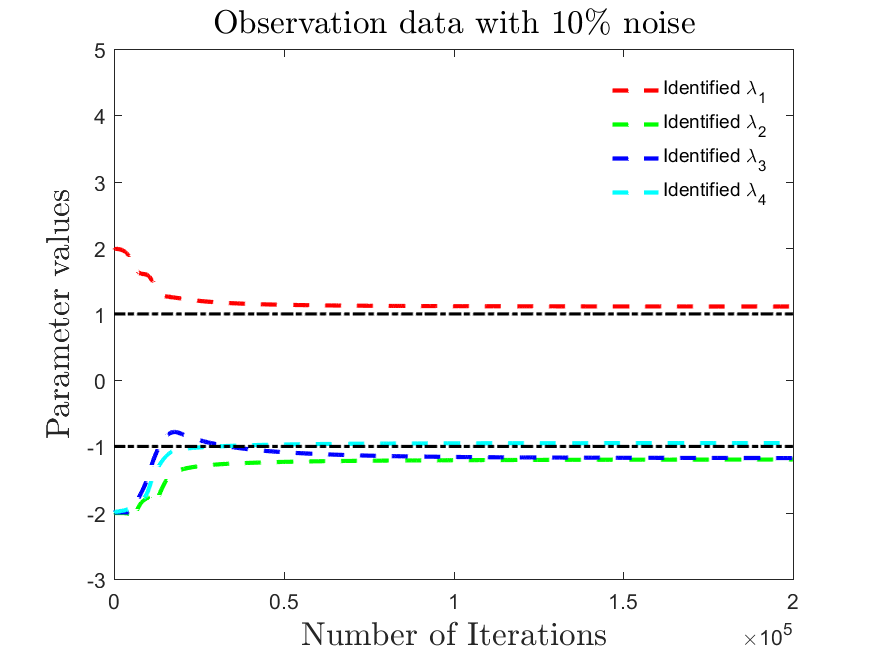}}
\end{minipage}
\hfill
\begin{minipage}[]{0.2 \textwidth}
 \leftline{~~~~~~~\tiny\textbf{(b1)}}
\centerline{\includegraphics[width=3.4cm,height=3.4cm]{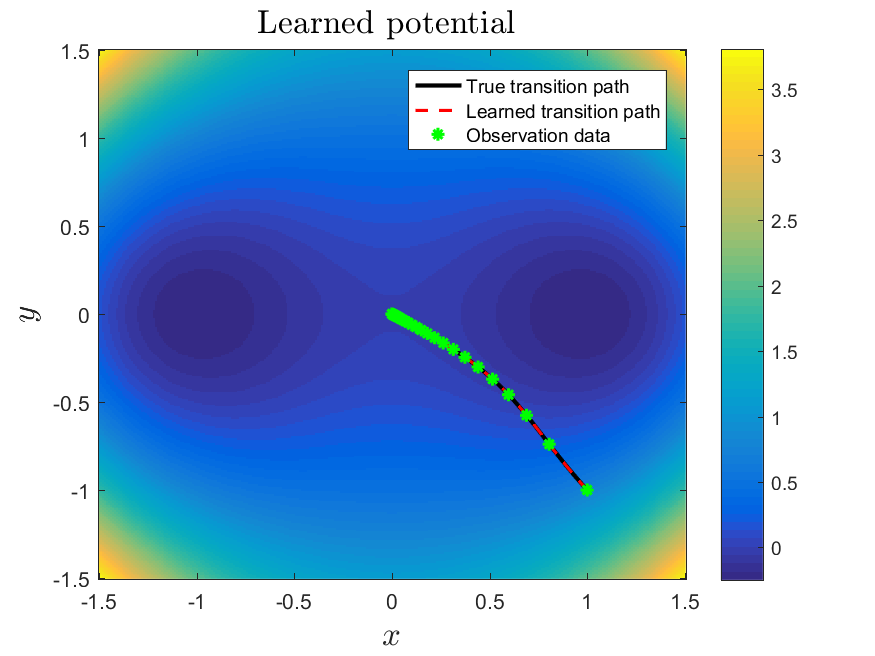}}
\end{minipage}
\hfill
\begin{minipage}[]{0.2 \textwidth}
 \leftline{~~~~~~~\tiny\textbf{(b2)}}
\centerline{\includegraphics[width=3.4cm,height=3.4cm]{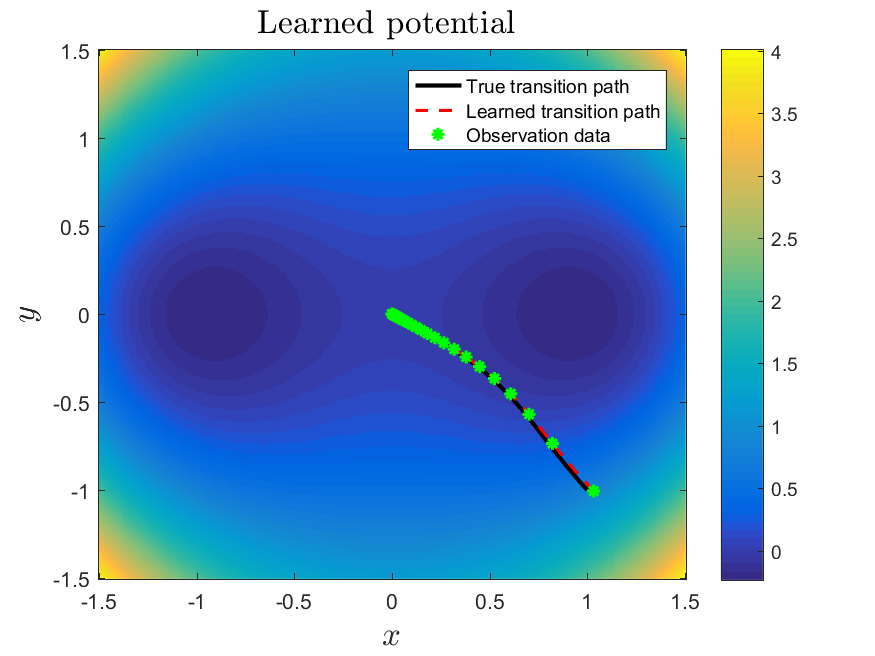}}
\end{minipage}
\hfill
\begin{minipage}[]{0.2 \textwidth}
 \leftline{~~~~~~~\tiny\textbf{(b3)}}
\centerline{\includegraphics[width=3.4cm,height=3.4cm]{2d_che/OM/2d_che_inverse_ob51_potential_num_2per_OM.png}}
\end{minipage}
\hfill
\begin{minipage}[]{0.2 \textwidth}
 \leftline{~~~~~~~\tiny\textbf{(b4)}}
\centerline{\includegraphics[width=3.4cm,height=3.4cm]{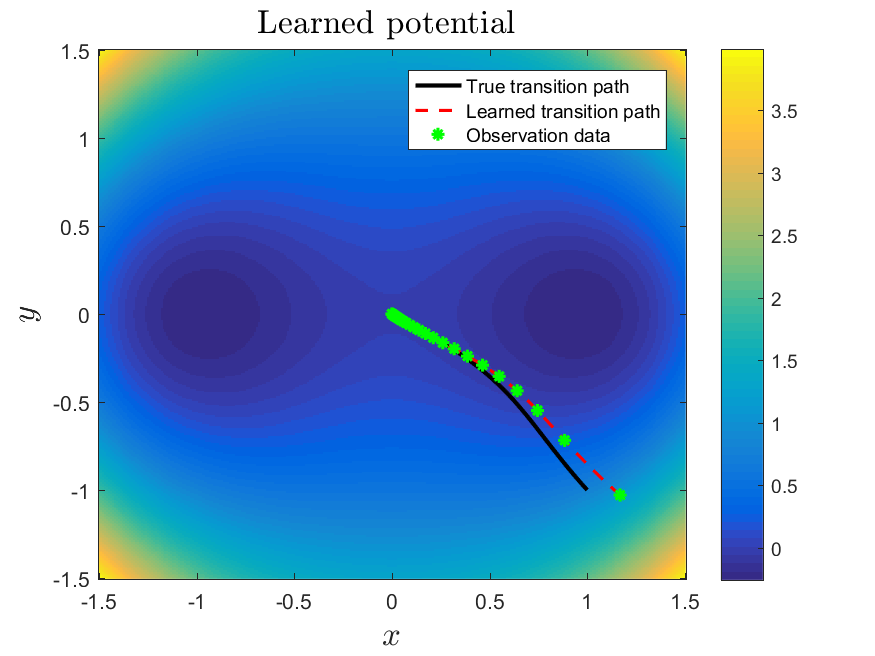}}
\end{minipage}
\caption{\textbf{Parametric estimation for transition time $T=10$ in Onsager-Machlup case ($\sigma =0.1$)  - stochastic Maier-Stein model:} (a1,b1) no noise; (a2,b2) 2\% noise; (a3,b3) 5\% noise; (a4,b4) 10\% noise. Above: learned parameters; below: learned the most probable transition pathway and potential function}
\label{2D_che_OM}
\end{figure}

\section{Conclusion}

We have investigated the transition phenomena of stochastic differential equations. The Onsager-Machlup action functional and Freidlin-Wentzell action functional provide a good framework to quantify the probability of the solution trajectories in a path tube, which offers information about transitions. We use physics-informed neural networks to calculate the most probable transition pathway via the Euler-Lagrange equation in both the Onsager-Machlup and Freidlin-Wentzell frameworks. A convergence result of the physics-informed neural networks for the Euler-Lagrange equation is proved. More precisely, we bound the expected loss in terms of the empirical loss and show the convergence result of the empirical loss.

In order to identify the stochastic differential equation from the observation data, we use a neural network to extract the drift function. The observation data is simulated by the Markovian bridge process, which corresponds to the stochastic differential equation. For a comparison, we also use the most probable transition pathway, computed by the neural network, as the observation data. The numerical experiments show our method could recover the stochastic differential equation well.


\appendix
\section{Detailed proof of the theorems}
In the following, we present the detailed proofs of Theorem 3.1 and Theorem 3.2. We also investigate the most probable transition pathway and the inverse problem of the stochastic double-well system.



\subsection{Proof of Theorem 3.1}
\begin{lem}
Suppose that m is large enough satisfying, for any $t\in U=(0,T)$, there exists $t'\in U$, such that $|t-t'|<\delta$. Then, we obtain
\begin{equation}
\operatorname{Loss}^{\operatorname{PINN}}(h;\lambda)\leq Cm\delta \cdot\operatorname{Loss}_m^{\operatorname{PINN}}\left(h;\lambda\right) + 3(C_L^2+1)\delta^{2\alpha} R(h),
\end{equation}
where $R(h)=[h]_{0,\alpha;U}+[h]_{1,\alpha;U}+[h]_{2,\alpha;U}$.
\end{lem}

\begin{proof}
By Assumption 1, for any $t\in U$, there exists $t'\in U$, such that $|t-t'|<\delta$, we then obtain
\begin{equation}
\begin{split}
&\ \ \ \ \|[\ddot{h}-g(\dot{h},h)](t)\|^2\\
&\leq  \|[\ddot{h}-g(\dot{h},h)](t)-[\ddot{h}-g(\dot{h},h)](t')\|^2+\|[\ddot{h}-g(\dot{h},h)](t')\|^2\\
&\leq  3\left(\|\ddot{h}(t)-\ddot{h}(t')\|^2+C_L^2\|\dot{h}(t)-\dot{h}(t')\|^2+C_L^2\|h(t)-h(t')\|^2\right)+\|[\ddot{h}-g(\dot{h},h)](t')\|^2\\
&\leq 3(C_L^2+1)\delta^{2\alpha}R(h) + \|[\ddot{h}-g(\dot{h},h)](t')\|^2.
\end{split}
\end{equation}
where $R(h)=[h]_{0,\alpha;U}+[h]_{1,\alpha;U}+[h]_{2,\alpha;U}$. For $t_j\in\mathcal{T}^m$, let $A_{t_j}$ be the Voronoi cell associated with $t_j$, i.e.,
\begin{equation}
A_{t_{j}}=\left\{t \in U \mid |t-t_{j}|=\min _{t^{\prime} \in \mathcal{T}^{m}}|t-t^{\prime}|\right\}.
\end{equation}
Let $\gamma_j=\mu(A_{t_j})$. Thus, $\sum_{j=1}^m\gamma_j=1$. Let $\gamma^{*}=\operatorname{max}_j\gamma_j$. Then, we obtain
\begin{equation}
\begin{split}
\|\ddot{h}-g(\dot{h},h)\|_{L^2(U;\mu)}^2&\leq \sum\limits_{j=1}^m\gamma_j\|[\ddot{h}-g(\dot{h},h)](t_j)\|^2+3(C_L^2+1)\delta^{2\alpha} R(h)\\
&\leq \gamma^{*}\sum\limits_{j=1}^m\|[\ddot{h}-g(\dot{h},h)](t_j)\|^2+3(C_L^2+1)\delta^{2\alpha} R(h).
\end{split}
\end{equation}
Let $P^{*}=\operatorname{max}_{t\in U}(B_{\delta}(t)\cap U)$, where $B_{\delta}(t)$ is a closed interval centered at $t$ with radius $\delta$. Since for each $t\in U$, there exists $x'\in \mathcal{T}^m$, such that $|t-t'|\leq\delta$, by the definition of Voronoi cell $A_{t_j}$, for each $j$, the cell $A_{t_j}$ is included in some closed ball $B_{\delta}$. Hence, it follows from Assumptions,
\begin{equation}
\begin{split}
\gamma^{*}\leq P^{*}\leq C\delta.
\end{split}
\end{equation}
Therefore, we have 
\begin{equation}
\begin{split}
&\ \ \ \ \operatorname{Loss}^{\operatorname{PINN}}\left(h;\lambda\right)\\
&=\lambda_r\|[\ddot{h}-g(\dot{h},h)\|_{L^2(U;\mu)}^2+ \frac{\lambda_b}{2}(\|h(0)-x_0\|^2+\|h(T)-x_T\|^2)\\
&\leq \gamma^{*}\lambda_r\sum\limits_{j=1}^m\|[\ddot{h}\!-\!g(\dot{h},h)](t_j)\|^2\!+\!3(C_L^2\!+\!1)\delta^{2\alpha} R(h)\!+\!\frac{\lambda_b}{2}(\|h(0)\!-\!x_0\|^2\!+\!\|h(T)\!-\!x_T\|^2)\\
&\leq Cm\delta \cdot\operatorname{Loss}_m^{\operatorname{PINN}}\left(h;\lambda\right) + 3(C_L^2+1)\delta^{2\alpha} R(h),
\end{split}
\end{equation}
which is the desired result of Lemma A.1.
\end{proof}

\noindent \text{\emph{Proof of Theorem 3.1.}} Let $\mathcal{T}^m$ be independently  and  identically  distributed  samples from probability distribution $\mu$ on $U=(0,T)$. The Assumption 1 yields the probabilistic space filling arguments. Thus by Lemma B.2 in \cite{shin2020convergence}, with probability $(1-\sqrt{m}(1-1/\sqrt{m})^m)$ at least, for each $t\in U$, there exists $t'\in\mathcal{T}^m$, such that $|t-t'|\leq c^{-1}m^{-\frac{1}{2}}$. 

Let $\delta=c^{-1}m^{-\frac{1}{2}}$ in Lemma A.1. We obtain, with probability $(1-\sqrt{m}(1-1/\sqrt{m})^m)$ at least,
\begin{equation}
\begin{split}
\operatorname{Loss}^{\operatorname{PINN}}\left(h;\lambda\right)\leq \frac{C}{c}m^{\frac{1}{2}} \cdot\operatorname{Loss}_m^{\operatorname{PINN}}\left(h;\lambda\right) + \frac{3(C_L^2+1)}{c^{2\alpha}}\cdot\frac{1}{m^{\alpha}} R(h).
\end{split}
\end{equation}
Recalling that $R(h)=[h]_{0,\alpha;U}+[h]_{1,\alpha;U}+[h]_{2,\alpha;U}$, we thus have
\begin{equation}
	\begin{split}
		\operatorname{Loss}^{\operatorname{PINN}}\left(h;\lambda\right)\leq \frac{C}{c}m^{\frac{1}{2}} \cdot\left(\operatorname{Loss}_m^{\operatorname{PINN}}\left(h;\lambda\right) + \lambda_m^R R(h)\right),
	\end{split}
\end{equation}
where $\lambda_m^R=\frac{3(C_L^2+1)}{C\cdot c^{{2\alpha-1}}}\cdot m^{-\alpha-\frac{1}{2}}$. The proof of Theorem 3.1 is complete.

\subsection{Proof of Theorem 3.2}
\noindent \text{\emph{Proof of Theorem 3.2.}} Since $h_m$ be a minimizer of the H\"{o}lder regularized loss $\operatorname{Loss}_m^{\operatorname{PINN}}\left(h_m;\lambda,\lambda^R\right)$ and $h_m, z_m^{*}\in \mathcal{H}_m$, by Assumption 2, we obtain
\begin{equation}
\begin{split}
\operatorname{Loss}_m^{\operatorname{PINN}}\left(h_m;\lambda,\lambda_m^R\right)&\leq \operatorname{Loss}_m^{\operatorname{PINN}}\left(z_m^{*};\lambda,\lambda_m^R\right)\\
&\leq \operatorname{Loss}_m^{\operatorname{PINN}}\left(z_m^{*};\lambda\right) + \lambda_m^R \cdot R^{*}.
\end{split}
\end{equation}
On the other hand,
\begin{equation}
\begin{split}
\operatorname{Loss}_m^{\operatorname{PINN}}\left(h_m;\lambda,\lambda_m^R\right)\geq \lambda_m^R R(h_m).
\end{split}
\end{equation}
Note that $\operatorname{Loss}_m^{\operatorname{PINN}}(z_m^{*};\lambda)=\mathcal{O}(m^{-\alpha-\frac{1}{2}})$ and $\lambda_m^R=\mathcal{O}(m^{-\alpha-\frac{1}{2}})$. Thus, combining (A.8) and (A.9), we have $\operatorname{Loss}_m^{\operatorname{PINN}}\left(h_m;\lambda,\lambda_m^R\right)=\mathcal{O}(m^{-\alpha-\frac{1}{2}})$. 

Then, by Theorem 3.1, we have that with probability $1-\sqrt{m}(1-1/\sqrt{m})^m$,
\begin{equation}
\begin{split}
\operatorname{Loss}^{\operatorname{PINN}}\left(h_m;\lambda\right)=\mathcal{O}(m^{-\alpha}), 
\end{split}
\end{equation} 
which is the deserved result.
\section{Calculus of Euler-Lagrange equations}
We will derive Euler-Lagrange equations \eqref{el_om} and \eqref{el_fw} of the Onsager-Machlup action functional and the Freidlin-Wentzell action functional, respectively. We set the action functional as 
\begin{equation}
\begin{split}\label{appact}
S(z,\dot{z})=\frac{1}{2}\int_0^TL(z,\dot{z})dt,
\end{split}
\end{equation}
where the Lagrangian $L(\cdot,\cdot)$ is the integrand of the Onsager-Machlup action functional \eqref{om} or the Freidlin-Wentzell action functional \eqref{fw}. The Euler-Lagrage equation of the action functional is
\begin{equation}
\begin{split}
\frac{d}{dt}\frac{\partial}{\partial \dot{z}}L(z,\dot{z})=\frac{\partial}{\partial z}L(z,\dot{z}).
\end{split}
\end{equation}

\noindent {\bf Onsager-Machlup Case}

Denote $A=(\sigma\sigma^T)^{-1}=(a_{ij})_{d\times d}$. In the Onsager-Machlup case, the Lagragian is 
\begin{equation}
\begin{split}
L^{OM}(z,\dot{z})&=\frac{1}{2}\left[\dot{z}-f(z)\right][\varepsilon^2\sigma\sigma^T]^{-1}[\dot{z}-f(z)]+\frac{1}{2}\nabla \cdot f(z)\\
&=\frac{1}{2\varepsilon^2}(\dot{z}^i-f^i(z))a_{ij}(\dot{z}^j-f^j(z))+\frac{1}{2} \partial_jf^j(z).
\end{split}
\end{equation}
Here, we use the Einstein sum. Firstly, taking the partial differential for $\dot{z}_k$, we obtain
\begin{equation}
\begin{split}
\frac{\partial L^{OM}(z,\dot{z})}{\partial \dot{z}_k}&=\frac{1}{2\varepsilon^2}a_{ij}(\dot{z}^i-f^i(z))\delta^{jk}+\frac{1}{2\varepsilon^2}a_{ij}(\dot{z}^j-f^j(z))\delta^{ik}\\
&=\frac{1}{2\varepsilon^2}(a_{kj}+a_{jk})(\dot{z}^j-f^j(z)).
\end{split}
\end{equation}
Thus, taking the time derivative, we have
\begin{equation}
\begin{split}\label{om1}
\frac{d}{dt}\frac{\partial L^{OM}(z,\dot{z})}{\partial \dot{z}_k}=\frac{1}{2\varepsilon^2}(a_{kj}+a_{jk})(\ddot{z}^j-\partial_if^j(z)\dot{z}^i).
\end{split}
\end{equation}
Then, we take the partial differential for $z_k$.
\begin{equation}
\begin{split}\label{om2}
\frac{\partial L^{OM}(z,\dot{z})}{\partial z_k}&=-\frac{1}{2\varepsilon^2}\partial_kf^i(z)(a_{ij}+a_{ji})(\dot{z}^j-f^j(z))+\frac{1}{2}\partial_k\partial_jf^j(z).
\end{split}
\end{equation}
Combining equations \eqref{om1} and \eqref{om2}, the Euler-Lagrage equation yields 
\begin{equation}
\begin{split}
\frac{1}{2\varepsilon^2}(a_{kj}+a_{jk})(\ddot{z}^j-\partial_if^j(z)\dot{z}^i)=-\frac{1}{2\varepsilon^2}\partial_kf^i(z)(a_{ij}+a_{ji})(\dot{z}^j-f^j(z))+\frac{1}{2}\partial_k\partial_jf^j(z).
\end{split}
\end{equation}
If the diffusion is a diagonal constant matrix, i.e., $\sigma=\operatorname{diag}\{\sigma_{11},...,\sigma_{dd}\}$, then $A=\operatorname{diag}\{\frac{1}{\sigma_{11}^2},...,\frac{1}{\sigma_{dd}^2}\}$. Therefore, the Euler-Lagrange equation corresponding to the Onsager-Machlup action functional \eqref{om} reduces to
\begin{equation}
\begin{split}
\ddot{z}_k=\dot{z}^j[\partial_jf^k(z)-\frac{\sigma_{kk}^2}{\sigma_{jj}^2}\partial_kf^j(z)]+\frac{\sigma_{kk}^2}{\sigma_{jj}^2}f^j(z)\partial_kf^j(z)+\frac{\varepsilon^2\sigma_{kk}^2}{2}\partial_k\partial_jf^j(z),
\end{split}
\end{equation}
for $k=1,...,d$. This is the deserved Euler-Lagrange equation as in equation \eqref{el_om}.

\noindent {\bf Freidlin-Wentzell Case}

In contrast to the Onsager-Machlup case, the Lagrangian misses the divergence part $\nabla \cdot f(z)$ in the Freidlin-Wentzell case. But the derivation is similar. In addition, if $\sigma$ is an identity matrix, the Euler-Lagrange equation corresponding to the Freidlin-Wentzell action functional \eqref{fw} reduces to
\begin{equation}
\begin{split}
\ddot{z}_k=\dot{z}^j[\partial_jf^k(z)-\partial_kf^j(z)]+f^j(z)\partial_kf^j(z),
\end{split}
\end{equation}
for $k=1,...,d$. This is the deserved Euler-Lagrange equation as in equation \eqref{el_fw}.



\bibliographystyle{plain}
\bibliography{references}

\end{document}